%% file: Arxiv_version.tex
\documentclass[11pt,a4paper]{article}
\usepackage{cancel}
\usepackage{graphicx} 
\usepackage{bm}
\usepackage{soul}
\usepackage{amsmath}
\usepackage{amssymb}
\usepackage{amsthm}
\usepackage{mathrsfs}
\usepackage{epsfig}
\usepackage{verbatim}
\usepackage{url}
\usepackage{color}
\usepackage[body={154mm,229mm},top=34mm]{geometry}
\usepackage{accents}
\usepackage{epstopdf}

\DeclareMathOperator*{\essinf}{\operatorname{essinf}}

\newcommand{\ninseps}[3]{
\begin{figure}[h]
\begin{center}
 \scalebox{#3}{\includegraphics{#1}}
\end{center}
\caption{\hspace{0.25cm}#2\label{f:#1}}
\end{figure}
}

\newtheorem{theorem}{Theorem}

\newtheorem{lemma}{Lemma}
\newtheorem{corollary}{Corollary}

\newtheorem{remark}{Remark}

\newtheorem{proposition}{Proposition}

\newcommand{\Ymin}{Y^{\min}}
\newcommand{\Zmin}{Z^{\min}}

\title{Backward Stochastic Differential Equations with Nonmarkovian Singular Terminal Values}

\author{
Ali Devin Sezer \thanks{Middle East Technical University, Institute of Applied Mathematics, Ankara, Turkey, {\tt devin@metu.edu.tr}
}
, Thomas Kruse \thanks{University of Duisburg-Essen, Thea-Leymann-Str. 9, 45127 Essen, Germany,
e-mail: {\tt thomas.kruse@uni-due.de}
}
, Alexandre Popier \thanks{Laboratoire Manceau de Math\'ematiques, Universit\'e du
Maine, Avenue Olivier Messiaen, 72085 Le Mans, Cedex 9, France.
e-mail: {\tt alexandre.popier@univ-lemans.fr}
}
}
\date{\today}

\begin{document}
\maketitle

\begin{abstract}
We solve a class of BSDE with a power function $f(y) = y^q$, $q > 1$,
driving its drift
and
with the terminal boundary condition
$ \xi = \infty \cdot \mathbf{1}_{B(m,r)^c}$ (for which $q > 2$ is assumed) or
$ \xi = \infty \cdot \mathbf{1}_{B(m,r)}$,
where $B(m,r)$ is the ball in the path space $C([0,T])$
of the underlying Brownian motion centered
at the constant function $m$ and radius $r$.
The solution involves the
derivation and solution of a related heat equation in which $f$
serves as a reaction term
and which is accompanied by
singular and discontinuous Dirichlet boundary conditions. 
Although the solution
of the heat equation is discontinuous at the corners of the domain the BSDE
has continuous sample paths with the prescribed terminal value.
\end{abstract}

\vspace{0.5cm}
\noindent {\bf AMS 2010 class:} 35K57, 35K67, 60G40, 60H30, 60H99, 60J65 

\section{Introduction}
One of the first points emphasized in an introductory
ordinary differential equations (ODE)
course is that the solution of an ODE
may explode in finite time;
the equation
\begin{equation}\label{e:ode0}
\frac{dy}{dt} = y^q,
\end{equation}
with $q > 1$ serves 
 as the primary example.
Indeed, specify the terminal value $y_T = \infty$
 to \eqref{e:ode0} and
\begin{equation}\label{e:trivialsol}
y_t \doteq ((q-1)(T-t))^{1-p},~~~ t < T,~~~ 1/p + 1/q = 1,
\end{equation}
will be the corresponding unique solution of \eqref{e:ode0}
($p$ is the  H\"older conjugate of $q$).
Now let $W$ be a standard Brownian motion and $\{{\mathscr F}_t\}$ be its natural filtration. For a terminal condition
$\xi \in {\mathscr F}_T$,
one can think of the backward stochastic differential equation (BSDE)
\begin{align}
Y_s &= Y_t + \int_s^t f(Y_r) dr + \int_s^t Z_r dW_r, 
~~0 < s < t < T, \label{e:SDE}\\
Y_T &= \xi, \label{e:terminalcondition}
\end{align}
$f(y) = -y|y|^{q-1}$, 
$Y$ {\em continuous}\footnote{
The requirement that $Y$ be
continuous on $[0,T]$ is natural- otherwise one could solve the SDE
\eqref{e:SDE}
arbitrarily  on $[0,T)$ and set $Y_T = \xi$; thus
the terminal condition $\xi$ would have no bearing on the
behavior of $(Y,Z)$ on $[0,T).$} on $[0,T]$, 
as a stochastic generalization / perturbation of the ODE\eqref{e:ode0}
because
for $\xi=\infty$ identically, one can set $Z_t = 0$ and
reduce \eqref{e:SDE} to \eqref{e:ode0} for which
$Y_t = y_t$ is the unique solution.
But $\xi$ is a random variable and
can also be chosen
equal $\infty$ over a measurable set $A \in {\mathscr F}_T$
and a finite random variable over $A^c.$
Can one solve the BSDE (\ref{e:SDE},\ref{e:terminalcondition})
with such terminal conditions? 
An analysis of this and related questions 
began with the article \cite{popier2006backward}, where $W$ is
assumed to be $d$-dimensional. 
\cite{popier2006backward}
 proved in particular that there exists
a pair of processes $(\Ymin,\Zmin)$ adapted to the filtration ${\mathscr F}_t$ satisfying \eqref{e:SDE}
and where $\Ymin$ satisfies almost surely (a.s.)
\begin{equation}\label{e:upperbound}
\lim_{t\rightarrow T} \Ymin_t  \ge \xi = \Ymin_T.
\end{equation}
In other words, the process $\Ymin$ is a continuous process on $[0,T)$, whose left-limit as $t$ goes to $T$ exists a.s. and dominates the terminal condition $\xi=\Ymin_T$. Moreover, 
$\Ymin$ of
\cite{popier2006backward} 
 is minimal: 
for any other pair $(\hat{Y}, \hat{Z})$ satisfying \eqref{e:SDE} and
\begin{equation}\label{e:super}
\liminf_{t\to T} \hat{Y}_t \geq \xi,
\end{equation}
one has
\begin{equation}\label{e:minprop}
\Ymin_t \leq \hat{Y}_t, \text{ a.s.}, t\in [0,T].
\end{equation}
Following \cite{krus:popi:16} 
we will refer to any pair satisfying \eqref{e:SDE} and \eqref{e:super} as a super-solution
of the BSDE (\ref{e:SDE},\ref{e:terminalcondition}).
Thus $(\Ymin,\Zmin)$ is the minimal super-solution of the BSDE (\ref{e:SDE},\ref{e:terminalcondition}). 
To strengthen \eqref{e:upperbound} to the a.s. equality 
\begin{equation}\label{e:cont_terminalcondition}
\lim_{t\rightarrow T} \Ymin_t = \Ymin_T =  \xi
\end{equation}
and hence solving
the BSDE (\ref{e:SDE},\ref{e:terminalcondition})
for general $\xi \in {\mathscr F}_T$ 
turns out to be a difficult problem.  
The article \cite{popier2006backward} proved \eqref{e:cont_terminalcondition}
for $\xi$ of the form $\xi = g(W_T)$, 
where the function
$g :{\mathbb R} \mapsto {\mathbb R}_+ \cup \{\infty\}$ satisfies $\{ g= \infty\}$ is closed and for any compact subset $K$ of $\{g < \infty\}$, ${\mathbb E}[g(W_T) \mathbf{1}_{K}(W_T)]< \infty.$
Because $\xi$ is a deterministic function of $W$, such terminal conditions are referred to as ``Markovian''. 
To the best of our knowledge, to delineate the class of $\xi \in {\mathscr F}_T$ for which the
BSDE (\ref{e:SDE},\ref{e:terminalcondition}) has a solution $Y$ on $[0,T]$ still remains
an open problem. 

\vspace{0.5cm}
The {\it goal of the present work} is to construct solutions
to the BSDE (\ref{e:SDE},\ref{e:terminalcondition}) 
for a class of non-Markovian final conditions $\xi \in {\mathscr F}_T$;
we will also prove that the solutions we construct are equal to the
minimal supersolutions $(\Ymin, \Zmin)$ of \cite{popier2006backward},
which will imply that \eqref{e:cont_terminalcondition} holds for
the terminal conditions we treat.
To the best of our knowledge, the present work is the first to derive
these types of results for non-Markovian singular terminal conditions.
The class of $\xi$ which we will focus on is best
explained using the canonical path space
$\Omega \doteq C([0,T ] , \mathbb{R})$,
the set of all $\mathbb{R}$-valued continuous paths 
$\omega$ on $[0,T]$, equipped with its sup norm 
\[
||\omega||_{\infty} \doteq \sup_{t\in[0,T]} |\omega(t)|,
\]
and a family of Wiener measures 
$\{ {\mathbb P}^x$, $x \in {\mathbb R}\}$,
under which
the canonical process $W_t(\omega) = \omega(t)$ is a standard Brownian
motion with initial condition $W_0 = x.$
As before $\mathbb{F} = ({\mathscr F}_t)_{0\leq t\leq T}$ is the canonical filtration generated by $W$.
Then 
${\mathscr F}_T$ is the 
the Borel field of $\Omega$ corresponding to
the sup norm $||\cdot||_\infty$ and the 
basic ${\mathscr F}_T$-measurable 
random variables are the indicator
functions of open /closed subsets of $C([0,T])$. The open and closed 
subset of $C([0,T])$
are generated by balls with respect to the norm $||\cdot||_\infty$ and the
simplest balls in turn are those centered around constant functions.
Thus we arrive at the class of terminal values we would like to cover:
\begin{equation}\label{e:newterm}
\xi_1 = \infty \cdot \mathbf{1}_{B(m,r)^c} \quad \mbox{or} \quad \xi_2 = \infty \cdot \mathbf{1}_{B(m,r)}
\end{equation}
where $B(m,r)$ is the ball $\{ \omega: ||\omega -m||_\infty \le r \}$, for some $m \in {\mathbb R}$ and $r> 0$. 
To simplify notation we will assume throughout that 
$m = r= L/2$ for some $L > 0$
for which the expressions for $\xi$ in \eqref{e:newterm} become
$\xi_1 = \infty \cdot \mathbf{1}_{B(L/2,L/2)^c}$ and
$\xi_2 = \infty \cdot \mathbf{1}_{B(L/2,L/2)}$;
all of what follows trivially extends
to arbitrary $m \in {\mathbb R}$ and $r > 0$.

The Markovian terminal 
conditions provide (via It\^o's formula) 
the connection between
BSDE and a class of semilinear / quasilinear parabolic PDE
\cite{pardoux1992backward}. 
In the case of singular terminal conditions of the type $g(X_T)$ where
$g$ can take the value $+\infty$, the associated
parabolic PDE is coupled with
singular boundary conditions; a considerable number of articles appeared over the last several
decades
(see \cite{bara:pier:84,brez:frie:83,marcus1999initial,dynk:kuzn:98,lega:96} and the references therein) studying the PDE 
\begin{equation}\label{eq:gene_PDE}
V_t - \Delta V + V^q = 0,
\end{equation}
where $\Delta$ denotes the Laplace operator, allowing for singular terminal values. 
The same PDE \eqref{eq:gene_PDE} is directly related to the BSDE (\ref{e:SDE},\ref{e:terminalcondition}) and will play a key role in our analysis below. 
See \cite[Section 4]{popier2006backward} for more 
on the link between the BSDE (\ref{e:SDE},\ref{e:terminalcondition}) and the PDE \eqref{eq:gene_PDE}. 
\vspace{0.5cm}

The main idea of the present paper 
for the solution of the BSDE (\ref{e:SDE},\ref{e:terminalcondition}) for $\xi$ of the form \eqref{e:newterm} is 
to reduce the question to a Markovian problem in the random time interval
$[0,\tau \wedge T]$ where
$\tau \doteq \{ t\in [0,\infty): W_t \in \{0,L\}\}$.
For $\tau < T$, the terminal conditions given in \eqref{e:newterm} 
reduce to constants
\[
\xi_1(\omega) =\infty \cdot \mathbf{1}_{B(m,r)^c} = \infty,~~
\xi_2(\omega)  =\infty \cdot \mathbf{1}_{B(m,r)} = 0,
\]
and the SDE \eqref{e:SDE} reduces to the ODE \eqref{e:ode0} on $(\tau ,T]$.
Solving it  on $(\tau,T]$ with the terminal condition $\xi_1(\omega) = \infty$ 
gives the solution 
\[
Y_t^1 = y_t, ~Z_t^1 =0,~~~ t \in (\tau,T], 
\]
of the BSDE (\ref{e:SDE},\ref{e:terminalcondition}) on $(\tau, T]$ for
$\xi=\xi_1.$ Similarly, solving the same ODE on the same time interval
with 
the terminal condition $\xi_2(\omega) = 0$ gives the solution
\[
Y_t^2 = 0,~ Z_t^2=0,~~~ t \in (\tau,T],
\]
of the same BSDE for $\xi=\xi_2.$
These then give the value of the solutions $Y^i$ at time $\tau < T$:
\begin{equation}\label{e:intermidateterminal1}
Y^1_\tau = y_\tau, Y^2_\tau = 0.
\end{equation}
On  the set $T < \tau$, the terminal conditions $\xi_1$ and $\xi_2$
reduce to
\begin{equation}\label{e:intermidateterminal2}
\xi_1(\omega) =0, ~~
\xi_2(\omega)  = \infty.
\end{equation}
Next we solve the same BSDE in the time
interval $[0, T\ \wedge \tau]$ 
using \eqref{e:intermidateterminal1} 
and \eqref{e:intermidateterminal2} 
as terminal
conditions. Thus our BSDE is reduced to one with a Markovian
terminal condition at the random terminal time $\tau \wedge T$. 
Now It\^o's formula
provides the connection between the solution of the reduced BSDE to the solution of the
parabolic equation
\begin{equation}\label{e:PDEfirst}
\partial_t V + \frac{1}{2}\partial^2_{xx} V - V^q = 0;
\end{equation}
\eqref{e:intermidateterminal1} 
and
\eqref{e:intermidateterminal2} 
suggest the following boundary conditions to accompany the PDE:
\begin{equation}\label{e:boundary1}
V(0,t)  = V(L,t) = y_t, t \in [0,T],~~
V(x,T) = 0, 0  < x < L
\end{equation}
for $\xi_1$ and 
\begin{equation}\label{e:boundary2}
V(0,t)  = V(L,t) = 0, t \in [0,T],~~
V(x,T) = \infty, 0  < x < L
\end{equation}
for $\xi_2.$ 
Proposition \ref{p:BSDEsol} of Section \ref{s:first_case}
gives the details of the above reduction.

With these steps our problem is reduced to the solution of the PDE
\eqref{e:PDEfirst} and the boundary condition \eqref{e:boundary1} for
$\xi_1$ and the boundary condition \eqref{e:boundary2} for $\xi_2.$
The main difficulty with the solution of these equations are the 
discontinuous (at the corners $(0,T),(L,T) \in {\mathbb R}^2$)
and infinite valued boundary conditions.
The most relevant
work that we have identified in the literature on
the solution of \eqref{e:PDEfirst} and the boundary conditions
\eqref{e:boundary1} and \eqref{e:boundary2} 
is \cite{marcus1999initial},
which contains results giving the existence of  weak solutions to the 
PDE \eqref{e:PDEfirst} in $d$ space dimensions when coupled with boundary
conditions which are allowed to take the value $+\infty.$
However, 
these results occur in
\cite{marcus1999initial} in the context of the computation of
initial traces and
within a general framework where boundary conditions
and solutions are specified in a weak Sobolev-sense;
to treat these questions the authors of \cite{marcus1999initial} 
use PDE and analysis 
results developed by them over a number of works.
We think that one can build an argument starting from results in
\cite{marcus1999initial}
to get a classical solution to 
(\ref{e:PDEfirst},\ref{e:boundary1})
and
(\ref{e:PDEfirst},\ref{e:boundary2})
 having the regularity and the boundary continuity
properties needed for our purposes but this appears to be a nontrivial task.
In this paper, we follow a different route
and give a new self contained construction 
of classical solutions of 
(\ref{e:PDEfirst},\ref{e:boundary1}) 
and
(\ref{e:PDEfirst},\ref{e:boundary2}) 
starting
from classical parabolic PDE theory with smooth boundary conditions 
\cite{friedman} and building on it
using smooth approximation from below of the boundary conditions 
and elementary probabilistic techniques.

Once the solution of the BSDE is built as above, the last step
is to
connect them with the corresponding minimal supersolution
$(\Ymin,\Zmin)$;
this is achieved by an argument using the approximating sequence
of functions constructed in the solution of the PDE.

One change in the application of the above steps
to the terminal conditions $\xi_1$ and $\xi_2$ is the assumption
we make on $q$: for $\xi_1 $ we need $q > 2$ whereas $q>1$ suffices
for $\xi_2$. This is coupled with the following change in the argument:
for $q > 2$, the classical heat equation $\partial_t V + \frac{1}{2}
\partial^2_{xx} V = 0$ also has a classical solution $v_0$ with the boundary
condition \eqref{e:boundary1}. In the treatment of $\xi_1$ we use $v_0$
as an upper bound in constructing an approximating sequence for the
solution of \eqref{e:PDEfirst} and \eqref{e:boundary1}, which
ensures the continuity of the limit of the approximation at the
boundaries. For $\xi_2$ the
corresponding boundary condition is \eqref{e:boundary2}, for which $v_0$
doesn't exist (regardless of the value of $q$) 
but we are able to construct an upperbound directly working with
the PDE \eqref{e:PDEfirst} and the boundary condition \eqref{e:boundary2}
and for this $q >1$ suffices. Other than this, the arguments for
$\xi_1$ and $\xi_2$
are the same. To reduce repetition and shorten the paper we give them
in detail for the first case in Section \ref{s:first_case}, the necessary
changes for $\xi_2$ are given in Section \ref{s:secondtermcond}.
The results of these sections are summarily given in Theorems
\ref{thm:first_case} (Section \ref{s:first_case}) and 
\ref{thm:second_case} (Section \ref{s:secondtermcond}).
Both of these sections present numerical examples (graphs of functions
and example sample paths) of the 
constructed solutions of the BSDE and those of the associated PDE.

We would like to note a connection between our results
and the BSDE theory with $L^p$ terminal conditions.
The assumption $q >2$ for $\xi_1$ implies that,
with the above reduction
of the BSDE (\ref{e:SDE},\ref{e:terminalcondition}) to the
random time interval $[0,\tau\wedge T]$, 
the reduced terminal condition will be in $L^1$; thus
one can also invoke
the existence results of \cite{bria:dely:hu:03} to
construct a solution for the terminal condition $\xi_1$.
The reduction to the time interval $[0,\tau\wedge T]$ doesn't lead to
an  $L^1$ terminal condition for $\xi = \xi_2$;
the PDE approach above applies to both $\xi_1$ and $\xi_2.$

A well known fact in the prior literature (see, e.g.,\cite{krus:popi:16}) is the link between
the BSDE (\ref{e:SDE},\ref{e:terminalcondition}) and the following stochastic 
optimal control problem:  
the controlled process $C$ is $C_s = c + \int_s^t \alpha_s ds$, the 
running cost is $|\alpha|^p$ and the terminal cost is $|C_T|^{p} \xi$, where $0 \cdot \infty = 0$. 
The random variable $\xi$ is a penalty on the terminal value of $C$; in particular the controlled process is 
constrained to satisfy $C_T = 0$ if $\xi = +\infty$. A growing number of articles study variants 
and generalizations of this control problem (with $\xi= \infty$ identically)  with applications to liquidation of 
portfolios of assets, see \cite{ankirchner2014bsdes, graewe2013smooth,graewe2015non,krus:popi:16}. 
The value function 
$v$ of the control problem is given by the minimal solution $\Ymin$: $v(t,x) = |x|^p \Ymin_t$ \cite{krus:popi:16}.
Therefore, our results in Section \ref{s:first_case} and \ref{s:secondtermcond} give explicit expressions
for the value function of this control problem for 
$\xi = \infty \cdot \mathbf{1}_{B(m,r)}$ and
$\xi = \infty \cdot \mathbf{1}_{B(m,r)^c}$.
Section \ref{s:control} uses this connection to derive estimates on the
conditional probabilities ${\mathbb P}(B(m,r) | {\mathscr F_t})$ and
${\mathbb P}(B(m,r)^c | {\mathscr F_t})$, $ t \in [0,T).$

Let us point out further prior literature on the solution of the BSDE (\ref{e:SDE},\ref{e:terminalcondition}):
\cite{popier2006backward} considers the case where $\xi$ is a function $g(X_T)$ where $X$ is the solution of a forward SDE 
\[
X_t  = x + \int_0^t b(s,X_s)dr  + \int_0^t \sigma(s,X_s) dW_s;
\]
(for the assumptions on $b$, $\sigma$ and $q$ we refer the reader to \cite{popier2006backward}). Since then, two works \cite{mato:pioz:popi:16,popier2016limit} appeared treating the BSDE (\ref{e:SDE},\ref{e:terminalcondition})
both focusing on $\xi$ of the form $g(X_T)$. The work \cite{mato:pioz:popi:16} extends the results of \cite{popier2006backward} to the class of backward doubly stochastic SDE (BDSDE in short). 
The article \cite{mato:pioz:popi:16} proves under these models that a minimal super-solution $(\Ymin,\Zmin)$ exists which is also continuous at the terminal time $T$ with $\Ymin_T = \xi = g(X_T)$. 
The work \cite{popier2016limit} also considers the BSDE 
with three additional extensions a) there are an additional jump term given 
by a Poisson random measure; b) the drift term $f(Y_s) = Y_s^q$ 
in \eqref{e:SDE} is replaced with a general $f$ satisfying a number of 
conditions which includes as a special case the function $y\rightarrow y^q$ 
and c) it works with a general complete right continuous filtration to which 
all of the given processes are adapted (as in \cite{krus:popi:16}); 
\cite{popier2016limit} proves that under these model assumptions that the 
minimal super-solution $\Ymin$ to the BSDE is continuous at the terminal time 
with $\Ymin_T = \xi = g(X_T)$ (in \cite{popier2016limit} jump terms are 
also allowed in the dynamics of $X$). Note that existence and minimality 
of $(\Ymin,\Zmin,U^{\min},M^{\min})$ were proved already in 
\cite{krus:popi:16} (the terms $U^{\min}$ and $M^{\min}$ come from the 
Poisson measure and the general filtration).
A recent work treating integro-partial differential generalizations of 
\eqref{eq:gene_PDE} with singular terminal conditions is 
\cite{popier2016integro}, which contains many further references and 
a literature review on parabolic PDE with singular boundary conditions, 
their connections to BSDE and their probabilistic solutions.

We indicate several directions for future research in the Conclusion.
\section{A first non-Markovian case} \label{s:first_case}

{This section implements 
for the terminal condition
$\xi =\xi_1 = \infty \cdot \mathbf{1}_{B(m,r)^c}$
the argument whose outline was given
in the introduction.}
We will denote by $D$ the domain $(0,L) \times (0,T)$. 
For $x \not\in (0,L)$, 
$\mathbb{P}^x(\xi = +\infty) = 1$ or $\mathbb{P}^x(\xi=0)=1$
and the problem becomes trivial for such $x$ (the same comment applies
to the terminal condition 
$\xi = \infty \cdot \mathbf{1}_{B(L/2,L/2)}$ as well).
Therefore, will assume the initial condition $x$
to satisfy $x \in (0,L)$; none of the arguments of the present work
depend on the initial point $W_0 = x$
beyond this consideration, thus for ease of notation we will
simply write ${\mathbb P}$ for ${\mathbb P}^x$ and always assume
$x \in (0,L).$ We summarize the results of this section
in the following Theorem.

\begin{theorem} \label{thm:first_case}
If $q>2$ then there is a function $u$ which is $C^\infty$ in the $x$ variables and $C^1$ in the $t$ variable and continuous on 
$\bar{D} \setminus \{(L,T),(0,T)\}$ satisfying the PDE \eqref{e:PDEfirst} with the 
boundary condition \eqref{e:boundary1}
such that
\begin{enumerate}
\item
\begin{equation}\label{e:defYZ}
Y_t = \begin{cases} u(W_t,t)&, t < \tau \wedge T,\\
			 y_t&, \tau \le t \le T,
	  \end{cases}~~~
Z_t = \begin{cases} u_x(W_t,t)&, t < \tau \wedge T,\\
			 0&, \tau \le t \le T.
	  \end{cases}
\end{equation}
solve the BSDE (\ref{e:SDE}, \ref{e:terminalcondition}) 
with
$\xi = \xi_1 = \infty \cdot \mathbf{1}_{B(L/2,L/2)^c}$; in particular,
$Y$ is continuous on $[0,T]$,
\item 
We have 
$(\Ymin,\Zmin)= (Y,Z)$; in particular
\eqref{e:cont_terminalcondition} holds.
\end{enumerate}
\end{theorem}
\begin{proof}
Proposition \ref{p:BSDEsol} of subsection \ref{ssect:reduc}
proves that
given any classical solution $u$ of \eqref{e:PDEfirst} and
the boundary condition \eqref{e:boundary1}, the processes
$(Y,Z)$ defined as in \eqref{e:defYZ} satisfy the BSDE (\ref{e:SDE},\ref{e:terminalcondition}) and the $Y$ process is continuous on $[0,T]$.
Proposition \ref{p:solvePDE} of subsection \ref{ssect:solvepde} constructs
a classical solution $u$ of \eqref{e:PDEfirst} and the boundary
condition \eqref{e:boundary1}. Finally,  Proposition \ref{p:YeqYmin} proves $Y = \Ymin$
for the $u$ constructed in Proposition \ref{p:solvePDE},
which implies in particular that, for $\xi = \xi_1$, \eqref{e:cont_terminalcondition} holds.
\end{proof}

\begin{remark}{\em
As pointed out in the introduction,
the connection between the BSDE (\ref{e:SDE},\ref{e:terminalcondition})
and the PDE \eqref{e:PDEfirst} is well known 
for Markovian terminal conditions.
The above result says that the same connection continues to hold
when ones uses the non-Markovian $\xi_1$
as terminal condition for the BSDE.
}
 \end{remark}

{We give several numerical examples and simulation of our results in subsection \ref{ss:numerical}.}

\subsection{Reduction to heat equation with reaction} \label{ssect:reduc}

As outlined in the introduction, our approach to solving the BSDE (\ref{e:SDE},\ref{e:terminalcondition})
$\xi = \infty \cdot \mathbf{1}_{B(L/2,L/2)^c}$
will be by breaking the problem into
two random time intervals $[0,\tau \wedge T)$ and
$(\tau \wedge T, T]$; on the latter the problem reduces to the trivial \eqref{e:ode0} with
the terminal value $y_T = \infty.$ The value of the unique solution $y_\tau$ at $\tau$ then provides
the terminal condition over the interval 
 $[0,\tau \wedge T)$; thus we end up with a Markovian problem and can attack it via the associated PDE.
These are the main ideas underlying the next proposition.
\begin{proposition}\label{p:BSDEsol}
Suppose $u:{\bar D} \rightarrow {\mathbb R}$ is
 $C^\infty$ in the $x$ variable and continuously differentiable in the $t$ variable over $D$,
continuous on $\bar{D} \setminus \{(L,T),(0,T)\}$
and satisfies the PDE \eqref{e:PDEfirst} and the boundary condition \eqref{e:boundary1} in the classical sense.
Then the pair $(Y,Z)$ of \eqref{e:defYZ} satisfies the BSDE {\rm (\ref{e:SDE}, \ref{e:terminalcondition})} and is continuous on $[0,T]$.
\end{proposition}

\begin{proof}
We begin by proving that $Y$ is continuous on $[0,T].$
First consider the case $\{\tau < T\}.$
By assumption $u$ is continuous on $\bar{D} \setminus \{(L,T),(0,T)\}$. 
Therefore, $u$ is continuous on $[0,L] \times [0, \tau]$,
$[0,\tau] \varsubsetneq [0,T]$.
In addition, $W$ has continuous sample paths. 
Then $t \mapsto u(W_t,t)$ is the composition of two continuous maps on $[0,\tau]$ and therefore is a continuous function on that interval. 
On the other hand, by definition \eqref{e:defYZ}
$Y_t =y_t$  for $t > \tau$; 
and the continuity of $t\mapsto y_t$ on $[\tau,T]$ implies the same for $Y$; 
finally the continuity of $Y$ at $\tau$ follows from the boundary condition \eqref{e:boundary1} and the definition of $Y$ given
in \eqref{e:defYZ}: $u(W_\tau,\tau) = y_\tau = Y_\tau$.
Thus we see that $Y$ is continuous on $[0,T]$ on the set $\{\tau < T\}$. 
The event $\{\tau = T\}$ is of measure zero, thus it only
remains to consider the case $\{ \tau > T \}.$ 
By definition \eqref{e:defYZ}
$Y_t = u(W_t,t)$, $t \in [0,T]$ for $\omega \in \{ \tau > T \}.$
The continuity of the sample path of $W$ and the compactness
of $[0,T]$ imply that  there exists $\delta > 0$ such that
\begin{equation}\label{e:strictinclusion}
W_t(\omega) \in [\delta, L-\delta], t \in [0,T],
\end{equation}
for $\omega \in \{\tau(\omega) > T \}.$ 
By assumption $u$ is continuous
on $[\delta, L -\delta] \times [0,T]$.
Then $t\mapsto Y_t= u(W_t,t)$, $t \in [0,T]$ is the composition of two
continuous functions and hence continuous.
This proves the continuity of $Y$ on $[0,T]$.

By definition 
\begin{align*}
Y_T &= y_T \cdot {\bm 1}_{\{\tau < T\}} + u(W_T,T){\bm 1}_{\{\tau > T\}}\\
    &= \infty \cdot {\bm 1}_{B(L/2,L/2)^c} + u(W_T,T){\bm 1}_{\{\tau > T\}}.
\end{align*}
The fact
\eqref{e:strictinclusion} and that $u$ satisfies \eqref{e:boundary1}
imply $ u(W_T,T){\bm 1}_{\{\tau > T\}} = 0.$ This and the last display imply
$Y_T = \xi_1$, i.e., that $Y$ satisfies the terminal condition \eqref{e:terminalcondition} with $\xi =\xi_1.$

It remains to prove that for fixed $ s < t < T$ \eqref{e:SDE} holds almost surely.
On the set $\{ \tau \le  s \}$, $Y_r = y_r$ and
$Z_r = 0$ for $r  \in [s,t]$ and \eqref{e:SDE} reduces to
\[
y_t = y_s + \int_s^{t} y^q_r dr,
\]
which is equivalent to \eqref{e:ode0} of which $y$ is a solution;
this establishes that \eqref{e:SDE} holds over $\{\tau \le s\}.$
Recall that by assumption, $u$ is smooth in $x$,
continuously differentiable in $t$ in $D$ and continuous on $\overline{D}\setminus
\{(L,T),(0,T)\}.$ In particular, $u$ is continuous on any $[0,L]\times [0,t]$
for $t < T$.
On the set $\{\tau > s\}$ apply It\^o's formula
to $u(W_r,r)$ between $s$ and $\tau \wedge t$ to get
\begin{align*}
 Y_{\tau \wedge t} = 
 Y_s + \int_s^{t \wedge \tau} \partial_x(W_r,r)dW_r
+ \int_s^{t\wedge \tau} \partial_t u(W_r,r)dr +\frac{1}{2}\int_s^{t \wedge \tau}
\partial_{xx} u(W_r,r)dr.
\end{align*} 
That $u$ satisfies \eqref{e:PDEfirst} implies
\begin{align}
 Y_{\tau \wedge t}&=  Y_s + \int_s^{t \wedge \tau} \partial_x u(W_r,r)dW_r
+ \int_s^{t\wedge \tau}  u^q(W_r,r)dr \notag\\
&= 
 Y_s + \int_s^{t \wedge \tau}  Z_r dW_r
+ \int_s^{t\wedge \tau}   Y^q_r dr, \label{e:uptotwedgetaun}
\end{align}
which implies \eqref{e:SDE} for $\{\tau > t\}.$ Finally, for
$\{ \tau \in (s,t) \}$:
\[
 Y_t =  Y_{\tau} + \int_\tau^t  Y_r^q dr.
\] 
Substituting the right side of \eqref{e:uptotwedgetaun} for $ Y_\tau$
in the last display gives
\[
 Y_t = 
 Y_s + \int_s^{t}  Z_r dW_r
+ \int_s^{t}   Y^q_r dr,
\]
where we have used $ Z_r = 0$ for $ r \in (\tau, t)$, which finishes
the proof that $( Y, Z)$ satisfies \eqref{e:SDE}.

\end{proof}

\subsection{Solution of the heat equation with reaction }\label{ssect:solvepde}
This subsection proves the key ingredient of Theorem \ref{thm:first_case}, i.e., 
the existence of a classical solution $u$ of \eqref{e:PDEfirst} and the boundary condition
\eqref{e:boundary1}.
Equation \eqref{e:PDEfirst} is often referred to as a reaction-diffusion equation where $V^q$ is the reaction term 
\cite[Example 1, page 535]{evans}. The main difficulty with \eqref{e:PDEfirst}  and \eqref{e:boundary1} is 
the discontinuity and unboundedness of the boundary condition near the corners $(L,T)$ and $(0,T)$ in ${\mathbb R}^2$. 
The next proposition asserts the existence of $u$ and gives its regularity properties (the function $v_0$ is defined in \eqref{e:defv0}).
Define
\begin{equation}\label{e:defv0}
v_0(x,t) \doteq {\mathbb E}_{x,t}
\left[ y_\tau \mathbf{1}_{\{ \tau < T \}} \right],
\end{equation}
where the subscript $(x,t)$ of the expectation operator denotes
conditioning on $W_t = x$.
The function $v_0$ will play a key role
in our construction of the solution $u$.
\begin{proposition}\label{p:solvePDE}
There is a unique function $ 0 \le u \le v_0$ which is $C^\infty$ in the $x$ variable
and continuously differentiable in the $t$ variable over $D$ and is continuous
on $\bar{D} \setminus \{(L,T),(0,T)\}$ and which solves \eqref{e:PDEfirst} and 
\eqref{e:boundary1}.
\end{proposition}
An intermediate step in the proof of Proposition \ref{p:solvePDE}
 will be to show that $v_0$ of \eqref{e:defv0} solves the classical heat equation
\begin{equation}\label{e:PDE0}
\partial _t V + \frac{1}{2} \partial^2_{xx} V  = 0, 
\end{equation}
over $D =(0,L) \times (0,T)$, with the same boundary condition
\eqref{e:boundary1} (see subsection \ref{ss:heat} below).
In this, the assumption $q >2$ and the following fact
will play a key role:
$q > 2$ implies that $-1 < 1-p <0$ and thus the solution \eqref{e:trivialsol} is integrable:
\begin{equation}\label{e:int_trivialsol}
\int_0^T y_s ds < \infty.
\end{equation}

Following notation parallel to that of \cite{friedman} define
\begin{align*}
B_{t_0}  \doteq  \{ (x,t), x \in (0,L) , t = t_0\}, ~~B\doteq   \{ (x,t), x \in (0,L), t = T \}, 
~~S  \doteq  \partial D \setminus \{B_0 \cup \bar{B} \};
\end{align*}
these sets are depicted in Figure \ref{f:thedomain}. 

\begin{figure}[h]
\begin{center}
\scalebox{0.8}{
\centerline{\input{thedomain}}}
\end{center}
\caption{\hspace{0.25cm}The domain and its boundaries\label{f:thedomain}}
\end{figure}
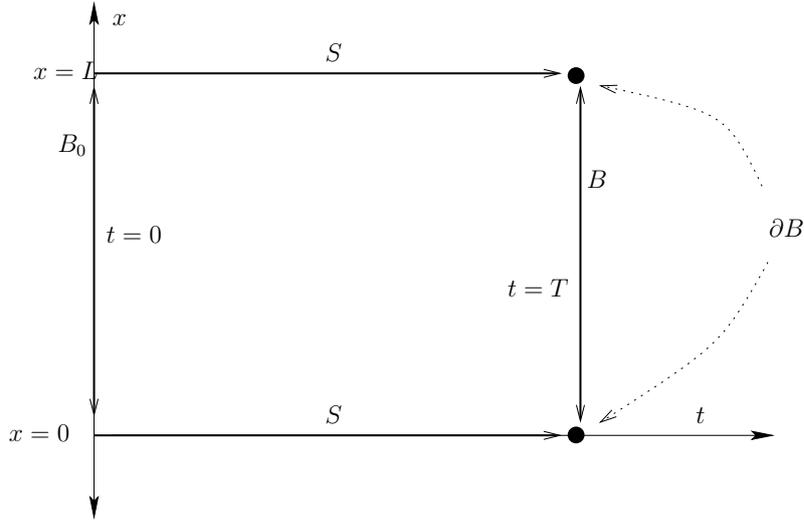

The proof of Proposition \ref{p:solvePDE} will proceed as follows:
\begin{enumerate}
\item $q > 2$ implies that $v_0$ solves (in the classical sense) the linear heat equation
\eqref{e:PDE0} and the boundary condition \eqref{e:boundary1} 
(see Lemmas \ref{l:v0} and \ref{l:regv0} in
subsection \ref{ss:heat} below),
\item Approximate
\eqref{e:boundary1} by a sequence of smooth boundary conditions to
which standard classical PDE theory
applies and yields classical solutions.
The solutions of the approximating equations are monotone in the
approximation parameter, and their limit is our candidate solution $u$. 
It\^o's formula implies an expectation representation 
for the approximate solutions. The solution $v_0$ of the heat equation in the first step 
gives us the necessary bound 
to invoke the dominated convergence theorem to infer that $u$ satisfies the same expected value representation 
as the prelimit functions (see Lemma \ref{l:defu} and \eqref{e:defu}).
\item Establish the regularity properties of $u$ (see Lemma \ref{l:regularityexistence}); 
we do this in two different ways.
The first approach relies only on probabilistic arguments and 
is elementary and direct, it uses the following elements:
a) explicit formulas for the density of the hitting time $\tau =\inf \{t: W_t \in \{0,L\}\}$ and 
the density of $W_t$ over sample paths restricted to stay in the interval $(0,L)$ upto time $t$ 
b) Duhamel's principle and c) the expected value representation of $u$. 
The second approach is based on analytic arguments for parabolic uniformly elliptic PDE. 
\item Establish the continuity properties of $u$ (Lemma \ref{l:continuityexistence}),
\item Once enough regularity is proved, the proof that $u$ actually solves the PDE follows from 
It\^o's formula, the expectation representation of $u$ and the strong Markov property 
of $W$.
\end{enumerate}
The above elements are put together in the Proof of Proposition \ref{p:solvePDE} given at the end of subsection \ref{ss:proof1}.

\subsubsection{Solution of the classical heat equation with singularities at the corners} \label{ss:heat}
The classical theory of Brownian
motion and of the classical heat equation
 suggest that $v_0$ is the unique solution
of \eqref{e:PDE0} and the boundary condition \eqref{e:boundary1}.
Let us prove that $v_0$ is finite and that it indeed
solves \eqref{e:PDE0} and \eqref{e:boundary1}.
Equation \cite[(4.1)]{douady1999closed}
(or It\^o's formula and direct computation) implies the following formula for the distribution function of $\tau$
conditioned on $W_t = x$:
\begin{equation}\label{e:disttau}
\mathbb{P}_{x,t}(\tau \le s )=1 + \mathbb{P}_{x,0}( W_{s-t} \in A^c )  - \mathbb{P}_{x,0}(W_{s-t} \in A )
\end{equation}
where $A \doteq \ \cup_{n \in {\mathbb Z} } \{ 2nL + [0,L] \}$,
$x \in [0,L]$
and $s > t.$
Substitute $A$ in \eqref{e:disttau} and change variables to rewrite
\eqref{e:disttau} as
\[
 \mathbb{P}_{x,t}(\tau \le s )=
1 + \sum_{n \in {\mathbb Z}}\frac{1}{\sqrt{2\pi}}
\left(
 \int_{\frac{(2n+1)L-x}{\sqrt{s-t}} }^{\frac{(2n+2)L -x}{\sqrt{s-t}}} e^{-y^2/2}dy  -
 \int_{\frac{2nL-x}{\sqrt{s-t}} }^{\frac{(2n+1)L -x}{\sqrt{s-t}}} e^{-y^2/2}dy \right).
\]
For $x \in (0,L)$, 
the derivative of the last display with respect to $s$
gives the density of $\tau$:
\begin{align}\label{e:densitytau}
f_\tau(x,t,s) \doteq\frac{(s-t)^{-3/2}}{\sqrt{2\pi}}  
\sum_{n \in {\mathbb Z}} 
((2n+1)L -x) e^{-\frac{((2n+1)L -x)^2}{2(s-t)}}
-
(2nL -x) e^{-\frac{(2nL -x)^2}{2(s-t)}}
\end{align}
(for $x \in \{0,L\}$, $\tau = t$ and $\mathbb{P}_{x,t}(\tau >s ) = 0$ identically
for $s > t$ and indeed the right side of \eqref{e:disttau} is identically
$0$ for $x \in \{0,L\}$);
Figure \ref{f:ftau} shows
the graph of $f_\tau$ for $t=2$, $L=4$, $x=3.5$ 
\ninseps{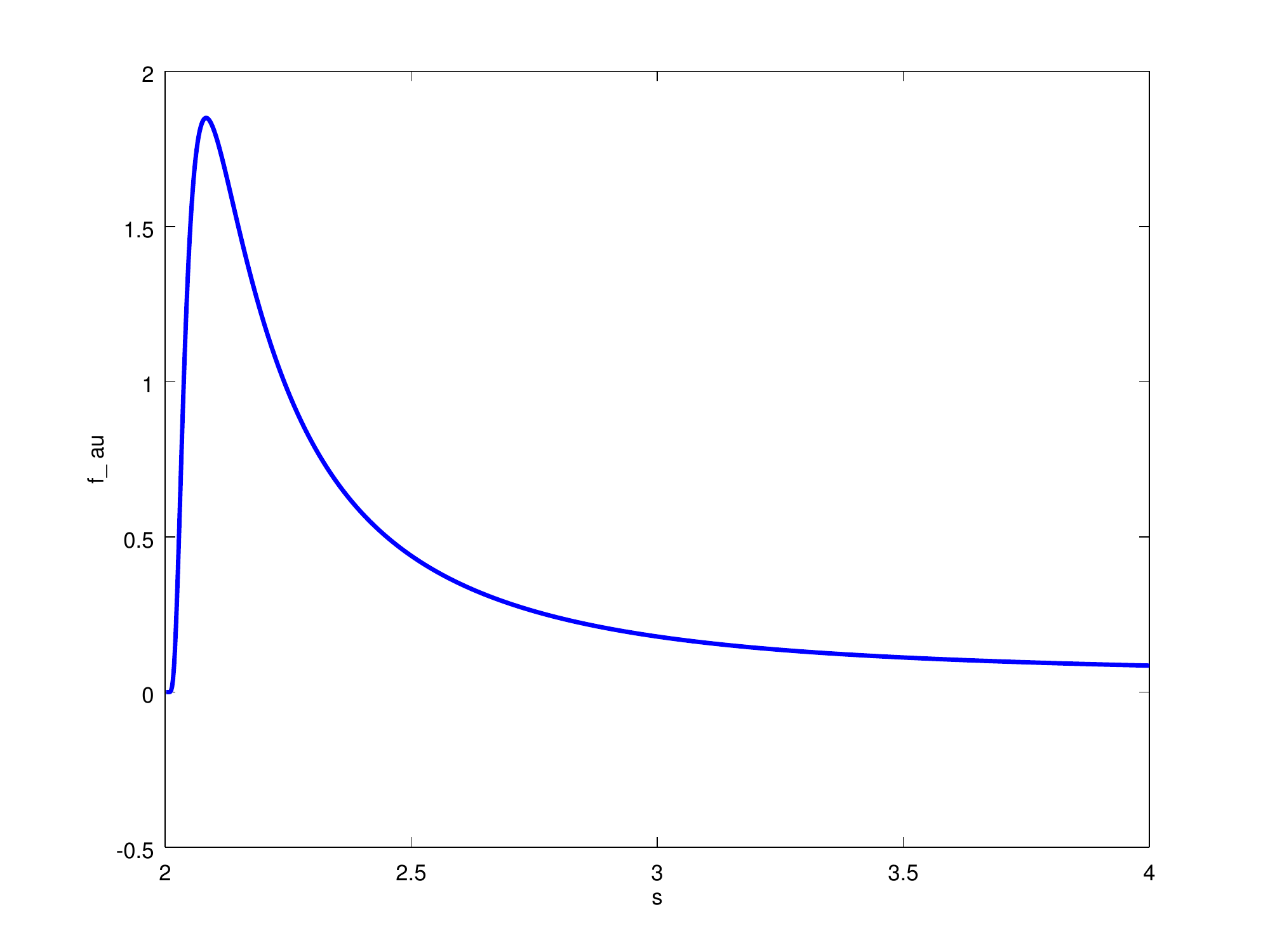}{The graph of $f_\tau$, $t=2$, $L=4$, $x=3.5$}{0.4}

For $(x,t) \in D$, write the expectation in \eqref{e:defv0} in terms of the density $f_\tau$:
\begin{equation}\label{e:v0}
v_0(x, t) = \int_t^{T} f_\tau(x,t,s) y_s ds = \int_t^{T} f_\tau(x,t,s) ((q-1)(T-s))^{1-p}ds.
\end{equation}
The formula \eqref{e:densitytau} and the behavior of $x\log(x)$ around $0$ imply that $f_\tau$ is continuous and smooth with continuous derivatives over the region
$[\delta_1,L-\delta_2] \times [t,\infty)$ for any $\delta_i > 0$
with $\delta_1 < L -\delta_2.$ Therefore from \eqref{e:int_trivialsol} we deduce that $v_0 (x,t)< \infty$ for $(x,t) \in D$ and $v_0$
 has the same regularity as $f_\tau$ in compact subsets of
$D$.
Integrability of $t\mapsto y_t$ \eqref{e:int_trivialsol}, the boundedness of $f_\tau$ in
compact subsets of $D$,
\eqref{e:v0} and the dominated convergence theorem also imply
$v_0(x,t) \rightarrow 0$ for $x \in (0,L)$ and $t\rightarrow T$.
Furthermore, for any $t < T$,  $f_\tau$ is continuous as a function
of $(x,s)$ on any compact
strip $[0,L] \times [T-\delta,T]$ as long as $t < T-\delta$.
This and \eqref{e:int_trivialsol}
imply
\begin{equation}\label{e:error0}
\left|v_0(x,t)  - \int_{t}^{T-\delta} f_\tau(x,t,s) y_s ds\right| \le \epsilon
\end{equation}
for any $\epsilon > 0$ when
$\delta > 0$ is small enough. Note
\[
\int_{t}^{T-\delta} f_\tau(x,t,s) y_s ds = 
{\mathbb E}_{x,t}\left[y_\tau
\mathbf{1}_{\{\tau < T-\delta\}}\right].
\]
$\mathbb{P}_{x,t}(\tau = T-\delta) = 0$, and $s\mapsto y_s \mathbf{1}_{\{s < T- \delta \}}$ 
is a continuous
and bounded function for $s \neq T-\delta.$ 
Now choose any sequence $(t_n,x_n) \rightarrow (t,x)$, $x \in \{\{0,L\}\}$
and $t < T.$
The law of the iterated logarithm \cite[Theorem 9.23]{MR1121940}
implies that
the hitting time
$\tau$ converges to $t$ 
as $n \rightarrow \infty.$
These imply 
\[
\lim_{n\rightarrow \infty} {\mathbb E}_{(x_n,t_n)}\left[y_\tau
\mathbf{1}_{\{\tau < T-\delta\}}\right] = y_t.
\]
This and \eqref{e:error0} imply
$v_0(x_n,t_n) \rightarrow y_t$.
Let us record what we have proved so far as a lemma:
\begin{lemma} \label{l:v0}
The function $v_0$ defined in \eqref{e:defv0} has the integral
representation \eqref{e:v0}, is smooth in $D$ (with continuous derivatives
of all orders in compact subsets of $D$) and continuous
on $\bar{D}\setminus \partial B$ and satisfies the boundary condition
\eqref{e:boundary1}.
\end{lemma}
Next we will use It\^o's formula and the regularity of $v_0$ to show
that in fact it is a solution to the heat equation \eqref{e:PDE0}.
\begin{lemma}\label{l:regv0}
$v_0$ solves \eqref{e:PDE0}.
\end{lemma}
\begin{proof}
Suppose there is $(x_0,t_0) \in D$ such that 
\begin{equation}\label{e:nonzeroassumption}
\frac{1}{2} \frac{\partial^2 v_0}{\partial x^2}(x_0,t_0) + \frac{ \partial v_0}
{\partial t} (x_0,t_0) \neq 0.
\end{equation}
Let $\delta > 0$, be so that $0 < x_0-\delta < x_0+ \delta < L$ and
$t_0 + \delta < T.$ By the previous proposition $v_0$ is smooth
on the compact set $N_{x_0} \doteq [x_0 - \delta, x_0 + \delta ] \times [t_0, t_0 + \delta]$
with continuous derivatives of all orders. Let $\tau_\delta$ be the first
time the process $(t,W_t)$ hits $\partial N_{x_0}$. 
By definition $\tau_\delta < \tau \wedge T.$ Conditioning on ${\mathscr F}_{\tau_\delta}$,
the strong Markov property of the Brownian 
motion
and the definition of $v_0$ imply
\begin{equation}\label{e:dynamicsprogamming}
v_0(x_0,t_0) = {\mathbb E}_{x_0,t_0}[v_0(W_{\tau_\delta}, \tau_\delta)].
\end{equation}
It\^o's formula applied to $v_0$ upto time $\tau_\delta$ gives
\begin{align}\label{e:contradiction}
&{\mathbb E}_{x_0,t_0}[
v_0(W_{\tau_\delta},\tau_\delta)] -
v_0(x_0,t_0) \\
&~~~=
{\mathbb E}_{x_0,t_0}\left[
\int_t^{\tau_\delta}\left( 
0.5 \frac{\partial^2 v_0}{\partial x^2}(W_s,s) 
+ \frac{ \partial v_0}{\partial t}(W_s,s)\right) ds \right].\notag
\end{align}
\eqref{e:dynamicsprogamming} implies that the left side of the last display
equals $0$. But the continuity of $0.5 \frac{\partial^2 v_0}{\partial x^2}+ 
\frac{ \partial v_0}{\partial t}$ on $N_{x_0}$ , $\tau_\delta \neq 0$
and \eqref{e:nonzeroassumption} imply that 
the right side of \eqref{e:contradiction} is nonzero, which is a contradiction.
Hence, \eqref{e:nonzeroassumption} cannot happen and $v_0$ indeed solves
\eqref{e:PDE0} in $D$.
\end{proof}

\subsubsection{Treating the $V^q$ term} \label{ss:proof1}

Equipped with the classical solution $v_0$ of the heat equation
\eqref{e:PDE0} and the boundary condition \eqref{e:boundary1}
we will proceed as follows to construct a classical
solution to \eqref{e:PDEfirst} and \eqref{e:boundary1}: define a
family of boundary conditions $y^{m,n}$ approximating $y$ (decreasing in $m$ and increasing in $n$) which are smooth upto
$\partial D$ satisfying the existence uniqueness results from the
classical theory of parabolic PDE \cite{friedman}. This gives us a family
of functions $u_{m,n}$, solving \eqref{e:PDEfirst} with boundary values
$y^{m,n}$ and which, by It\^o's formula,
have expected cost representations.
This, the dominated convergence theorem and \eqref{e:int_trivialsol} give, upon taking limits of $\{u_{m,n}\}$,
a candidate solution $u$, which also has the same expected cost representation as the prelimit functions $u_{m,n}.$ We will then use the expected cost 
representation of $u$ to improve our knowledge of $u$'s regularity.

The next lemma is a consequence of the maximum principle\footnote{The maximum principle also holds under much more weaker assumptions (see among others Lemma 2.7 in \cite{marcus1999initial} or Lemma 1.6 in \cite{marc:vero:01}).} and is well known for BSDE with monotone generator.
\begin{lemma}\label{l:monotonicity}
$\ $ 
\begin{enumerate}
\item Suppose $u_0 \ge 0$ and $u_1 \ge 0$ are two bounded smooth solutions of \eqref{e:PDEfirst} such that
$u_0|_{\partial D \setminus B_0} \ge u_1|_{\partial D \setminus B_0}.$ Then $u_0 \ge u_1$ on $D$.

\item Assume that $u_0$ is a continuous solution of \eqref{e:PDEfirst} with $|u_0|\leq K$ on $\partial D\setminus B_0$. Then $|u_0| \leq K$ on $\bar D$. 
\end{enumerate}

\end{lemma}
\begin{proof}
$ v = (u_0 - u_1)$ satisfies
\begin{equation}\label{e:PDEmon}
\partial_t v + \frac{1}{2} \partial^2_{xx} v - \frac{1}{2}R v = 0
\end{equation}
where $R = (u_0^q - u_1^q)/(u_0 - u_1) \mathbf{1}_{u_0 \neq u_1} > 0$  and the boundary condition $u_0 -u_1 \ge 0$ on $\partial D \setminus B_0.$ It\^o's formula implies 
\[
v(x,t) = {\mathbb E}_{x,t} \left[e^{-\int_t^{\tau \wedge T} R(W_s,s)ds} v(W_{\tau \wedge T},\tau \wedge T) \right] \ge 0.
\]
For the second claim of the lemma, we use the same estimate with $u_1=0$:
\[
u_0(x,t) = {\mathbb E}_{x,t} \left[e^{-\int_t^{\tau \wedge T} R(W_s,s)ds} u_0(W_{\tau \wedge T},\tau \wedge T) \right].
\]
\end{proof}

The next lemma identifies our candidate solution to the
PDE \eqref{e:PDE0} and the boundary condition \eqref{e:boundary1}.
\begin{lemma}\label{l:defu}
There exists a measurable function $0 \le u \le v_0$ which
satisfies
\begin{equation}\label{e:defu}
u(x,t) = {\mathbb E}_{x,t}\left[e^{-\int_t^\tau u^{q-1}(W_s,s)ds} y_\tau  
\mathbf{1}_{\{\tau < T\}}\right]
\end{equation}
or equivalently
\begin{equation}\label{e:defu_equiv}
u(x,t) = {\mathbb E}_{x,t}\left[-\int_t^{\tau \wedge T } u^q(W_s,s)ds +
y_\tau  \mathbf{1}_{\{\tau < T\}}  \right ],
\end{equation}
for $(x,t) \in D.$ 
\end{lemma}
\begin{proof}
Define 
\[
y^{(n)}_t \doteq y_{t - 1/n}.
\]
Hence for any $t\in [0,T]$, $|y^{(n)}_t| \le ((q-1)^{1-p}) n^{p-1}$.  
Define $\psi: \partial D \setminus B \rightarrow {\mathbb R}$ and $\psi_n: \partial D \setminus B \rightarrow {\mathbb R}$
as follows:
\begin{eqnarray*}
\psi(x,T) & =& \psi_n(x,T) = 0, \ x \in (0,L), \\
\psi(x,t) & = & y_t, \ (x,t) \in S, \\
\psi_n(x,t) & = & y^{(n)}_t, \ (x,t) \in S.
\end{eqnarray*}
The function $\psi$ describes exactly the boundary 
condition \eqref{e:boundary1}. Note that $\psi$ and $\psi_n$ are 
discontinuous at the corners $\partial B$ and $\psi_n \nearrow \psi$.
We will now approximate $\psi_n$ by a sequence of smooth $\psi_{m,n}$
so that we can invoke
\cite[Theorem 9, page 205]{friedman}. This result requires that
$\psi_{m,n} \in \bar{C}^{2+\delta}$ for $ \delta \in (\alpha,1)$,
where $\alpha$ is the H\"{o}lder constant associated with the boundary $S$,
and 
\begin{equation}\label{e:boundaryB}
\frac{\partial \psi_{m,n}}{\partial t} + \frac{1}{2} \frac{\partial ^2 \psi_{m,n}}{\partial x^2} 
- \psi_{m,n}^q = 0
\end{equation}
on $\partial B.$

To get the desired sequence, begin with two functions (linear in $x$):
\begin{align*}
\psi^{(0)}_{m,n}(x,t) &\doteq y^{(n)}_t \left[ 1 - x m/2  \right],\\
\psi^{(L)}_{m,n}(x,t) &\doteq y^{(n)}_t \left[1 - (L-x) m /2  \right].
\end{align*}
Let $\eta: {\mathbb R} \rightarrow [0,1]$, $\eta \in C^\infty$
be  as follows: $\eta' < 0$ on $(0,1)$,  $\eta \ge 0$, $\eta(x) = 1$ for $x \le 0$, 
$\eta(x) = 0$, for $x > 1$; one possible choice is
\[
\eta(x) =   \int_{(x \vee 0)\wedge 1}^{1} e^{\frac{-1}{1-(2y-1)^2}} dy  \Bigg / 
\int_{0}^{1} e^{\frac{-1}{1-(2y-1)^2}} dy.
\]
Now define
\[
\psi_{m,n}(x,t) = \psi^{(0)}_{m,n}(x,t) \eta\left(\frac{m^2 x -1}{m-1}\right) + \psi^{(L)}_{m,n}(x,t)\eta\left(\frac{m^2 (L-x) -1}{m-1}\right);
\]
for $ m > 2/L \vee 1.$ The resulting sequence $\psi_{m,n}$ of functions are nonnegative and smooth,
decreasing in $m$ with limit $\psi_n$ and
they all satisfy \eqref{e:boundaryB}. Figure \ref{f:psimn} shows the graph
of $\psi_{m,n}$ for $m=5$, $n=10$, $L=3$ and $T=1$. 
\begin{figure}[h] 
\begin{center}
 \includegraphics[width=0.6\textwidth]{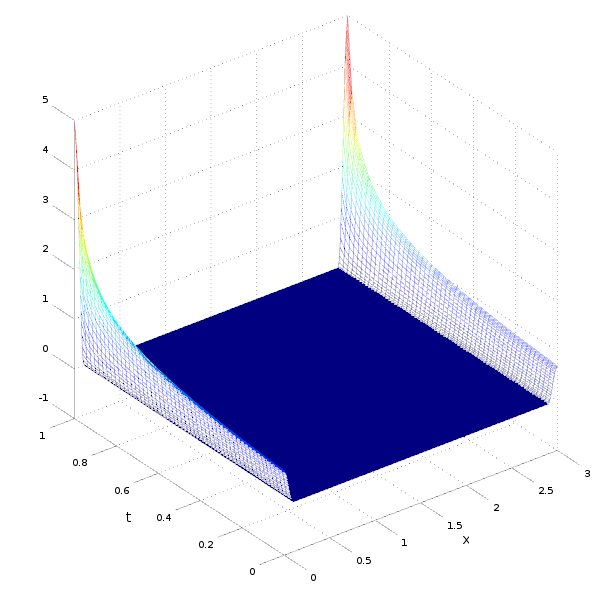}
\end{center}
\caption{\hspace{0.25cm} 
$\psi_{m,n}$, $m=10$, $n=50$, $L=3$, $T=1$\label{f:psimn}}
\end{figure}

Now we consider the PDE \eqref{e:PDEfirst} with the boundary condition
\begin{equation}\label{e:smoothbound}
V|_{\partial D \setminus B_0} = \psi_{m,n}.
\end{equation}
$|\psi_{m,n} | \le ((q-1)^{1-p}) n^{p-1} = \gamma_{n,q}$
 implies that it suffices for our purposes
to look for a solution $|u_{m,n}| \le  \gamma_{n,q}$ (see Lemma \ref{l:monotonicity}).  
The $f({\mathrm V})\doteq {\mathrm V}^q$ term in
\eqref{e:PDEfirst} is continuous in the interval 
$\left[0, \gamma_{n,q} \right]$ and
in particular it satisfies \cite[equation (4.10), page 203]{friedman}
\[
{\mathrm V} f({\mathrm V}) \le A_1 {\mathrm V}^2  + A_2,
{\mathrm V} \in \left[0, \gamma_{n,q} \right],
\]
for constants $A_1, A_2 \ge 0$,
and being monotone itself, it trivially satisfies 
\cite[equation (4.17), page 205]{friedman}, which requires $f$ be bounded by a
monotone function.
These imply that \cite[Theorem 9, page 205]{friedman} is applicable,
and therefore,
the PDE \eqref{e:PDEfirst} and the boundary condition \eqref{e:smoothbound}
have a classical solution $u_{m,n}$ 
that is continuous in $\bar{D}.$ It\^o's formula gives
\begin{align}\label{e:exprepumn}
u_{m,n} (x,t)
&= {\mathbb E}_{x,t}\left[
e^{-\int_t^{\tau \wedge T} u_{m,n}^{q-1}(W_s,s) ds} \psi_{m,n}(W_{\tau \wedge T},\tau \wedge T)  \right]\\
&= {\mathbb E}_{x,t}\left[
-\int_t^{\tau \wedge T} u_{m,n}^q(W_s,s) ds + \psi_{m,n}(W_{\tau \wedge T},\tau \wedge T) 
 \right].\notag
\end{align}
$\psi_{m,n} \ge 0$ and Lemma \ref{l:monotonicity} imply $u_{m,n} \ge 0.$
The functions $\psi_{m,n}$ are decreasing in $m$ and increasing in $n$ and they are all bounded;
this and
Lemma \ref{l:monotonicity} imply that
$u_{m,n}$  is decreasing in $m$ and increasing in $n$. Then we can define
\[
0 \le u_n \doteq \lim_{m \rightarrow \infty} u_{m,n}
\]
and
\[
0 \le u \doteq \lim_{n \rightarrow \infty} u_n.
\]
The representation \eqref{e:exprepumn}, $u_{m,n} \ge 0$, $\psi_{m,n} \le y_t$,
and
the definition \eqref{e:defv0} of $v_0$ imply $u_{m,n} \le v_0.$ Then by the above definitions
$u \le v_0.$
Now
the dominated convergence theorem (where $y_\tau \mathbf{1 }_{\{\tau < T\}}$ serves as the dominating function), 
$v_0 < \infty$, and taking limits in \eqref{e:exprepumn} give
\begin{align*}
u(x,t)
&= {\mathbb E}_{x,t}\left[
e^{-\int_t^{\tau \wedge T} u^{q-1}(W_s,s) ds} y_\tau  \mathbf{1}_{\{\tau < T \}} \right]\\
&= {\mathbb E}_{x,t}\left[
-\int_t^{\tau \wedge T} u^q(W_s,s) ds + y_\tau \mathbf{1}_{\{\tau < T\}} 
\right]. 
\end{align*}
\end{proof}

Our next task is to establish that $u$ is smooth in $D$.
\begin{lemma}\label{l:regularityexistence}
The function $u$ of \eqref{e:defu} is $C^\infty$ in $x$ and continuously differentiable in $t$ over $D$.  
\end{lemma}
We will give two different proofs for Lemma \ref{l:regularityexistence}.
The first is based on Duhamel's principle and uses
the density of $W_s$, $s> t$ on $[0,L]$ over the set 
$\{ 0 < W_u < L, \ \forall u \in [t,s] \}$; this proof is based
on fairly elementary calculations and in that sense direct.
We will define a number of functions
($U_1$, $U_2$, $U_3$ and $U_4$) in this proof, which will
also be used in the proof of the continuity of $u$ on $\bar{D}\setminus \partial B$
in Lemma \ref{l:continuityexistence} below.
The second proof uses general analytic results on the solution of uniformly 
elliptic parabolic PDE. 

For the first proof we need the density ${\mathbb P}(W_s \in dx, \tau> s)$ 
whose formula is given as  \cite[Equation (4.1)]{douady1999closed}; 
let us rederive it using our notation.  Parallel to \eqref{e:disttau} one 
first writes \begin{align}\label{e:distWTconstrained}
&\mathbb{P}_{x,t}( W_s \in (0,a), 0 < W_u < L, u \in (t,s) ) \\
&~~~= \mathbb{P}_{x,0}( W_{s-t} \in A_{a})
- \mathbb{P}_{x,0}(W_{s-t} \in B_{a}),\notag
\end{align}
for $ a \in (0,L)$ and where
$A_{a}\doteq \cup_{n \in {\mathbb Z}} \{ 2nL + (0,a) \}$
and $B_a \doteq \cup_{n \in {\mathbb Z}} \{ 2nL + (2L-a,2L) \}$;
the identities $A_L = A$ and  $B_L = A_L^c = A^c$ 
imply that \eqref{e:disttau} is a special case 
of \eqref{e:distWTconstrained}.
Substituting $A_a$ and $B_a$ in \eqref{e:distWTconstrained} and the normal
distribution of $W_{s-t}$ give
\begin{align*}
&\mathbb{P}_{x,t}( W_s \in (0,a), 0 < W_u < L, u \in (t,s) ) \notag\\
&~~=\sum_{n \in {\mathbb Z}}\frac{1}{\sqrt{2\pi}}
\left(
 \int_{\frac{2nL-x}{\sqrt{s-t}} }^{\frac{a + 2nL -x}{\sqrt{s-t}}} e^{-y^2/2}dy 
-
 \int_{\frac{2nL-a-x}{\sqrt{s-t}} }^{\frac{2nL -x}{\sqrt{s-t}}} e^{-y^2/2}dy 
\right).
\end{align*}
Differentiate the last display to get the density of $W_s$ on $(0,a)$
when the sample path of $W$ is constrained to stay in $[0,L]$ over the
time interval $[t,s]$:
\begin{equation}\label{e:densityWTconstrained}
f_W(x,t,s) \doteq \frac{1}{\sqrt{2\pi(s-t)}}
\sum_{n \in {\mathbb Z}} 
\left(
e^\frac{-(a+2nL-x)^2}{2(s-t)}
-
e^\frac{-(2nL-a-x)^2}{2(s-t)}\right).
\end{equation}
The above display implies that $f_W$ is smooth
for $ s > t$  and $x \in (0,L)$
in all variables with continuous derivatives of all orders.
Now we proceed with the first proof of Lemma \ref{l:regularityexistence}.
\begin{proof}
Write $u$ as the sum
\begin{equation}\label{e:split1}
u(x,t) = -U_1(x,t) + v_0(x,t),
\end{equation}
where
\[
U_1(x,t) \doteq {\mathbb E}_{x,t}
\left[ \int_t^{\tau \wedge T} u^q(W_s,s) ds\right],~~
v_0(x,t) = {\mathbb E}_{x,t}\left[  y_\tau \mathbf{1}_{\{\tau < T\}} \right].
\]
We already know that $v_0$ satisfies the conditions listed in the proposition.
It rests to show the same for $U_1$. First,
 $0 \le u \le v_0$ implies
\begin{equation}\label{e:boundU1}
0 \le U_1 \le v_0.
\end{equation}
Fix an arbitrary $T > \delta > 0$. We will now show that $U_1$
is smooth in $(0,L) \times (0,T-\delta)$, $\delta$ being arbitrary, this
will show $U_1$ is smooth on $(0,L) \times (0,T).$
For $t < T-\delta$,
the strong Markov property of $W$ and conditioning
on ${\mathscr F}_{\tau \wedge (T-\delta)}$ imply that we can write
$U_1$ in two pieces as follows:
\begin{align}
U_1(x,t) &= {\mathbb E}_{x,t} \left[ 
\int_t^{\tau \wedge (T-\delta) } u^q(W_s,s) ds\notag
+U_1(W_{T-\delta}, T-\delta ) \mathbf{1}_{\{ \tau > T- \delta \}} \right]\\
&= U_2(x,t) + U_4(x,t),
\label{e:split2}
\end{align}
where
\begin{align*}
U_2(x,t) &\doteq {\mathbb E}_{x,t} \left[ 
\int_t^{\tau \wedge (T-\delta) } u^q(W_s,s) ds\right],\notag\\
U_4(x,t) &\doteq {\mathbb E}_{x,t} \left[
U_1(W_{T-\delta}, T-\delta ) \mathbf{1}_{\{ \tau > T- \delta \}} \right].
\end{align*}
Let us write $U_4$ using the density $f_W$ given in
\eqref{e:densityWTconstrained}:
\[
U_4(x,t) = \int_0^L f_W(x,t,T-\delta,y) U_1(y,T-\delta) dy.
\]
That $0 \le v_0$ is continuous on $\bar{D} \setminus \partial B$ implies that
it is in particular bounded on $[0,L] \times [0,T-\delta].$
This and \eqref{e:boundU1} imply that $U_1(\cdot, T-\delta)$ is bounded by the same bound. This,
the existence and the continuity of the derivatives of $f_W$ in $x$ and
$t$ imply
that $U_4$ is smooth in $(0,L)\times [0,T-\delta]$
and is continuous on $[0,L]\times[0,T-\delta)$.  To study $U_2$ we will use
Duhamel's principle:
\begin{equation*}
U_2(x,t) = {\mathbb E}_{x,t} \left[ \int_t^{\tau\wedge(T-\delta)} u^q(W_s,s) ds\right]
= {\mathbb E}_{x,t} \left[ \int_t^{T-\delta} \mathbf{1}_{\{s < \tau \}} u^q(W_s,s) ds\right].
\end{equation*}
$v_0 \ge u \ge 0$ and Fubini's theorem imply
\begin{equation}\label{e:U2}
U_2(x,t)= \int_t^{T-\delta} 
{\mathbb E}_{x,t} \left[ \mathbf{1}_{\{s < \tau \}} u^q(W_s,s)\right]ds.
\end{equation}
Define
\begin{equation*}
U_3(x,t,s) \doteq 
{\mathbb E}_{x,t} \left[ \mathbf{1}_{\{s < \tau \}} u^q(W_s,s)\right]
\end{equation*}
and write \eqref{e:U2} in terms of $U_3$ (this is Duhamel's principle):
\begin{equation}\label{e:duhamel1}
U_2(x,t) = \int_t^{T-\delta} U_3(x,t,s) ds.
\end{equation}
The function $U_3$ can be written in terms of the density $f_W$ as
\begin{equation}\label{e:U3explicit}
U_3(x,t,s) = \int_0^L f_W(x,t,s,y) u^q(y,s) dy.
\end{equation}
Once again, for $s < T-\delta$, $0 \le  u^q$ is uniformly bounded above by
a constant. This, the smoothness of $f_W$ in $x$ imply that $U_3$ is smooth
in $x$ and $t$ on $(0,L) \times (0,T-\delta)$ for $t < s.$ $U_3$ is smooth in its $x$ variable, 
therefore $U_2$ is also
smooth in $x$ over the region $(0,L) \times (0,T-\delta).$ This, the
smoothness of $U_4$ and \eqref{e:split2} imply the same for $U_1$;
the smoothness of $U_1$ in $x$ and \eqref{e:split1} imply
the smoothness of $u$ in $x$.

Now we will derive the regularity of $u$ in the $t$ variable. Let us
begin with continuity of $U_2$ in $t$:
take any sequence $t_n \rightarrow t$, with $0 < t_n, t < T-\delta$
and $x \in (0,L).$
The continuity of $U_3$ in the $t$ variable implies that
the sequence of functions
\[
s \mapsto \mathbf{1}_{\{ s > t_n \}} U_3(x,t_n,s)
\]
converge almost surely to
\[
s\mapsto \mathbf{1}_{\{s > t \} } U_3(x,t,s)
\]
on the set $(0,T-\delta)$.
This and the bounded convergence theorem imply
\[
U_2(x,t_n) \rightarrow U_2(x,t),
\]
i.e., $U_2$ is also continuous in the $t$ variable on the set
$(0,L) \times(0,T-\delta).$
Thus we have: $U_2$, $U_4$ are both continuous on $(0,L) \times (0,T-\delta).$
This and \eqref{e:split2} imply that $U_1$ is continuous over the
same domain, this and \eqref{e:split1} imply the same for  $u$.
Now going back to \eqref{e:U3explicit} we see that this implies that
$U_3$ is also continuous in the $s$ variable. The continuity of $U_3$ 
in all of its variables,
\eqref{e:duhamel1}
and the fundamental theorem of calculus 
tell us that
$U_2$ is differentiable in $t$ and 
\begin{align*}
\frac{\partial U_2}{\partial t} &= -U_3(x,t,t)
 + \int_{t}^{T-\delta} \frac{\partial U_3}{\partial t}(x,t,s) ds\\
				&= -u^q(x,t) 
 + \int_{t}^{T-\delta} \frac{\partial U_3}{\partial t}(x,t,s) ds,
\end{align*}
which, in particular, is a continuous function on $(0,L) \times ( 0, T-\delta).$
Finally, this, the regularity of $U_4$ and \eqref{e:split2} imply
that $U_1$ is differentiable in $t$ with continuous derivative over
the domain $(0,L) \times (0,T-\delta)$, which in its turn, along with
\eqref{e:split1} imply the same for $u$.
This finishes the smoothness claims of the lemma on $u$. 
\end{proof}

We now give an alternative proof of the same lemma using
classical but deep results on parabolic PDE with a regularization bootstrap argument. 
We know from Lemma \ref{l:defu} that $u$, by construction, is the 
limit of a sequence $u_{m,n}$ of classical solutions of the PDE \eqref{e:PDEfirst} with the boundary 
condition \eqref{e:smoothbound}. The comparison principle (Lemma \ref{l:monotonicity}) implies that for any 
$m,n$ we have: $0\leq u_{m,n}(x,t) \leq y_t$ on $[0,L] \times [0,T]$. Thus the solutions are bounded from above by a 
function independent of $n$ and $m$. This will be useful in the proof below.

\begin{proof}[Second proof of Lemma \ref{l:regularityexistence}]
Now fix $\epsilon > 0$. On $[0,L] \times [0,T-\epsilon]$, $u_{m,n}$ is bounded (uniformly in $n$ and $m$) by $y_{T-\epsilon}$. Moreover the smooth function $u_{m,n}$ satisfies on $(0,L) \times (0,T-\epsilon)$
$$\partial_t u_{m,n} + \frac{1}{2} \partial^2_{xx} u_{m,n} = (u_{m,n})^q = f_{m,n}$$
with the H\"older continuous lateral boundary condition $y^{(n)}_t$ (and a bounded terminal condition $u_{m,n}(x,T-\epsilon)$). Here $f_{m,n}$ is a bounded function. We can apply \cite[Theorem III.10.1]{lady:solo:ural:68} 
(Conditions (1.2) and (7.1) of \cite{lady:solo:ural:68} are trivially satisfied in our setting). Therefore for any $\eta > 0$, $u_{m,n}$ is in $H^{\alpha,\alpha/2}([\eta,T-\eta]\times [0,T-\epsilon])$ (space of functions which $\alpha$-H\"older continuous in the space variable $x$ and $\alpha/2$-H\"older continuous in the times variable $t$). The value of $\alpha > 0$ and the H\"older norm of $u_{m,n}$ does not depend on $m$ and on $n$. In other words the H\"older norm of $u_{m,n}$ is bounded by some constant $C_{\alpha}$ depending only on $\eta$ and $\epsilon$. Moreover we already know that $u_{m,n}$ converges pointwise to $u_n$ (as $m$ goes to $+\infty$) and $u_n$ converges to $u$ (when $n$ tends to $+\infty$)\footnote{The Arzela-Ascoli theorem implies that $u_{m,n}$ (up to a subsequence) converges to some function $\widetilde u \in H^{\alpha,\alpha/2}([\eta,L-\eta]\times [0,T-\epsilon])$. Here $\widetilde u = u$ since pointwise convergence has been proved before.}. Therefore $u_n$ and $u$ are in $H^{\alpha,\alpha/2}([\eta,L-\eta]\times [0,T-\epsilon])$ and their H\"older norms are bounded by the same constant $C_\alpha$. 

Then $u_{m,n}$ is the solution of the same problem but now with more regular functions $f_{m,n}$ and $u_{m,n}(\cdot,T-\epsilon)$. Thus from 
\cite[Theorem IV.10.1]{lady:solo:ural:68}, we know that $u_{m,n}$ is in $H^{2+\alpha,1+\alpha/2}([\eta,L-\eta]\times [0,T-\epsilon-\eta])$ for any $\epsilon > 0$ and $\eta > 0$ and the norm estimates don't depend on $n$ and on $m$, but only on the H\"older norm of $f_{n,m}$ on $[\eta,L-\eta]\times [0,T-\epsilon-\eta]$ and the upper bound on $u_{m,n}$. Thus the same property holds for $u$. In other words $u$ is a classical solution on $(0,L)\times [0,T)$. This regularization argument can be iterated in order to obtain that $u$ is $C^\infty$.
\end{proof}

\begin{lemma}\label{l:continuityexistence}
$u$ of \eqref{e:defu} is continuous on $\bar{D} \setminus \partial B.$
\end{lemma}
\begin{proof}
Remember that $v_0$ is continuous on $B$ and takes the value $0$ there.
This and $0 \le u \le v_0$ imply the continuity of $u$ on $B$. 
By definition $u(x,t) = y_t = v_0(x,t)$ for $(x,t) \in S$.
We already know that $v_0$ is continuous on $S$. Furthermore,
by definition, $U_1 =0$ on $S$; these and \eqref{e:split1} imply
that it suffices to show 
\begin{equation}\label{e:contlast}
U_1(x_n,t_n) \rightarrow 0
\end{equation}
 for
$\{(x_n,t_n) \in D$, with $(x_n,t_n)\rightarrow (x,t) \in S.$
For this, we will use \eqref{e:split2} with $\delta > 0$ satisfying
$t < T-\delta$ and the definitions of
$U_2$ and $U_4.$ As $(x_n,t_n)\rightarrow (x,t) \in S$, 
$\tau \rightarrow t.$ This and the boundedness of $u^q$ on
$[0,L] \times [0,T - \delta]$ implies $U_2(x_n,t_n) \rightarrow 0.$
Lastly, $\tau\rightarrow t$ implies that $\mathbf{1}_{\{\tau > T-\delta\}}$
converges to $0$ almost surely. This and the boundedness of $U_1$ on
over $[0,L] \times [0,T - \delta]$ imply $U_4(x_n,t_n)\rightarrow 0$.
These and \eqref{e:split2} establish \eqref{e:contlast}.
\end{proof}

We can now complete the proof of Proposition \ref{p:solvePDE}:
\begin{proof}[Proof of Proposition \ref{p:solvePDE}]
By construction $u$ satisfies \eqref{e:boundary1}. 
Lemma \ref{l:regularityexistence} says that the function $u$ is 
smooth on $D$.
Thus It\^o's formula and the representation formula in Lemma 
\ref{l:defu} imply that $u$ satisfies \eqref{e:PDEfirst}; 
the details of a parallel argument 
have already been given in Lemma \ref{l:regv0} and are omitted.
Lemma \ref{l:continuityexistence}
says that $u$ is continuous on $\bar{D}\setminus \partial B$; \eqref{e:contlast},
$v_0(x,t) = y_t$ on $S$ and \eqref{e:split1} imply that $u(x,t) = y_t$
on $S$; $u \le v_0$ and $v_0 = 0$ on $B$ imply $u=0$ on $B$. These imply
that $u$ satisfies the boundary condition \eqref{e:boundary1}.

Next we prove the uniqueness claim, i.e., if $0 \le u_1 \le v_0$ is any other solution of the PDE \eqref{e:PDE0} and
the boundary condition \eqref{e:boundary1}, continuous on $\bar{D} \setminus \partial B$ then $u_1=u$ must hold. 
Proceeding as in the proof of Lemma \ref{l:monotonicity} define $v=u_1 -u$ and
$R = (u_1^q - u_0^q)/(u_1 - u) \mathbf{1}_{u_1 \neq u} > 0$. Now $v$ satisfies the PDE \eqref{e:PDEmon},
$v|_{\partial D} = 0$, and is continuous on $\bar{D} \setminus \partial B.$
These and It\^o's formula imply
\[
v(x,t) = 
{\mathbb E}_{x,t}\left[
e^{-\int_t^{T-1/n} R(W_s,s)ds} v\left(W_{T-\frac{1}{n}}, T-\frac{1}{n}\right)
\mathbf{1}_{\{\tau > T-1/n\}} 
                \right].
\]
The above display, $R \ge 0$, $|v| \le |u-u_1| \le 2v_0$ and Jensen's inequality imply
\[
|v(x,t)| \le {\mathbb E}_{x,t}\left[ 2v_0\left(W_{T-\frac{1}{n}}, T-\frac{1}{n}\right)
\mathbf{1}_{\{\tau > T-1/n\}} 
\right].
\]
The expectation representation \eqref{e:defv0} of $v_0$
implies that the second expectation above converges to $0$ with $n$. This proves $v=0.$

\end{proof}

\subsection{Connection to the minimal super-solution}\label{ss:completion}
It remains to establish the connection
between the solution of the BSDE constructed above and
the minimal supersolution $(\Ymin,\Zmin)$ of the BSDE (\ref{e:SDE},\ref{e:terminalcondition}) .
\begin{proposition}\label{p:YeqYmin}
Let $Y$ and $Z$ be the solution of the BSDE 
(\ref{e:SDE},\ref{e:terminalcondition}) defined in \eqref{e:defYZ}
where for $u$ we take the solution of the PDE \eqref{e:PDEfirst}
and the boundary condition \eqref{e:boundary1} constructed in Proposition 
\ref{p:solvePDE}. Then $(Y,Z) = (\Ymin,\Zmin).$
\end{proposition}
\begin{proof}
It follows from \eqref{e:SDE} and the definition of the It\^o integral 
that it suffices to prove $Y = \Ymin.$
The inequality $Y \ge \Ymin$ follows from the minimality property 
\eqref{e:minprop} of $\Ymin.$

To finish the proof we simply need to prove the converse inequality. Recall that $u$ is defined as the limit of a sequence of functions 
$u_n$ (see the proof of Lemma \ref{l:defu}). Using the same ideas as used in the proofs of Lemma \ref{l:regularityexistence}
one can prove that for any $n$, the function $u_n$ is smooth and is a classical solution of the PDE 
\eqref{e:PDEfirst} in $D$.
Moreover, as in the case of $u$, $u_n$ is continuous on $\overline{D} \setminus \partial B$ and satisfies the boundary condition:
\begin{equation*}
u_n(0,t)  = u_n(L,t) = y^{(n)}_t = y_{t-1/n}, t \in [0,T],~~
u_n(x,T) = 0, 0  < x < L.
\end{equation*}
Define 
\begin{equation*}
Y^n_t \doteq \begin{cases} u_n(W_t,t)&, t < \tau \wedge T,\\
			 y^{(n)}_t&, \tau \le t \le T,
	  \end{cases}~~~
Z^n_t \doteq \begin{cases} \partial_x u_n(W_t,t)&, t < \tau \wedge T,\\
			 0&, \tau \le t \le T.
	  \end{cases}
\end{equation*}
Straightforward modifications of previous arguments show that $(Y^n,Z^n)$ solves the BSDE \eqref{e:SDE} with terminal condition
$$\xi^n = y^{(n)}_T \mathbf{1}_{B(L/2,L/2)^c} = \left( \frac{n}{q-1}\right)^{\frac{1}{q-1}}\mathbf{1}_{B(L/2,L/2)^c}.$$
Since the minimal solution $(\Ymin,\Zmin)$ constructed in \cite{popier2006backward} is the increasing limit of solutions of the same BSDE but with terminal condition $\xi \wedge m$ (as $m$ goes to $+\infty$), 
the comparison principle implies 
 \[Y^n_t \leq \Ymin_t\]
a.s.  for any $t\in [0,T]$.
But $u_n$ converges to $u$ and from \eqref{e:defYZ} we obtain the desired inequality. 

\end{proof}

\subsection{Numerical examples} \label{ss:numerical}

Let us give several numerical examples for the PDE solutions constructed
above and the resulting solution $Y$ of the BSDE.
The left side of 
Figure \ref{f:u1} shows the graph of $u_{m,n}$ with $L=3$ and $T=1$,
$m=100$ and $n=50$ computed using a finite difference approximation
of the PDE with $\Delta x = 0.1$ and $\Delta t = 0.01.$
The right side of the same figure shows the graph of $u_{m,n}$ over the
line $x =  L/2 = 1.5$ for $m=100$ and $n=10$ and $n=150$ as well as the
graph of $y_t$; note $u_{100,10}(1.5,t) < u_{100,1000}(1.5,t) < y_t$ in the figure,
as expected.
Figure \ref{f:Y1} shows two randomly sampled sample paths of
the Brownian motion $W$ with $W_0 = L/2 = 3/2$
and the corresponding path for $Y$, computed
using \eqref{e:defYZ} where we use a numerical approximation
of $u_{m,n}$ with $m=100$ and $n=1000$ to approximate $u$.

\begin{figure}
\begin{center}
\includegraphics[width=0.5\textwidth]{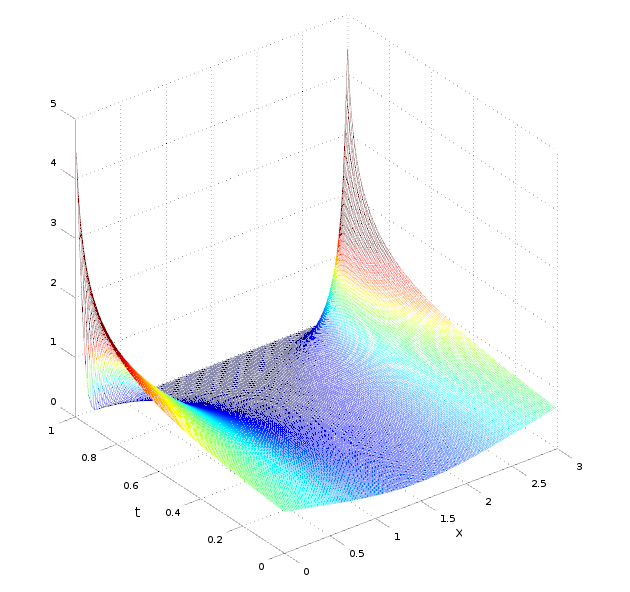}\includegraphics[width=0.5\textwidth]{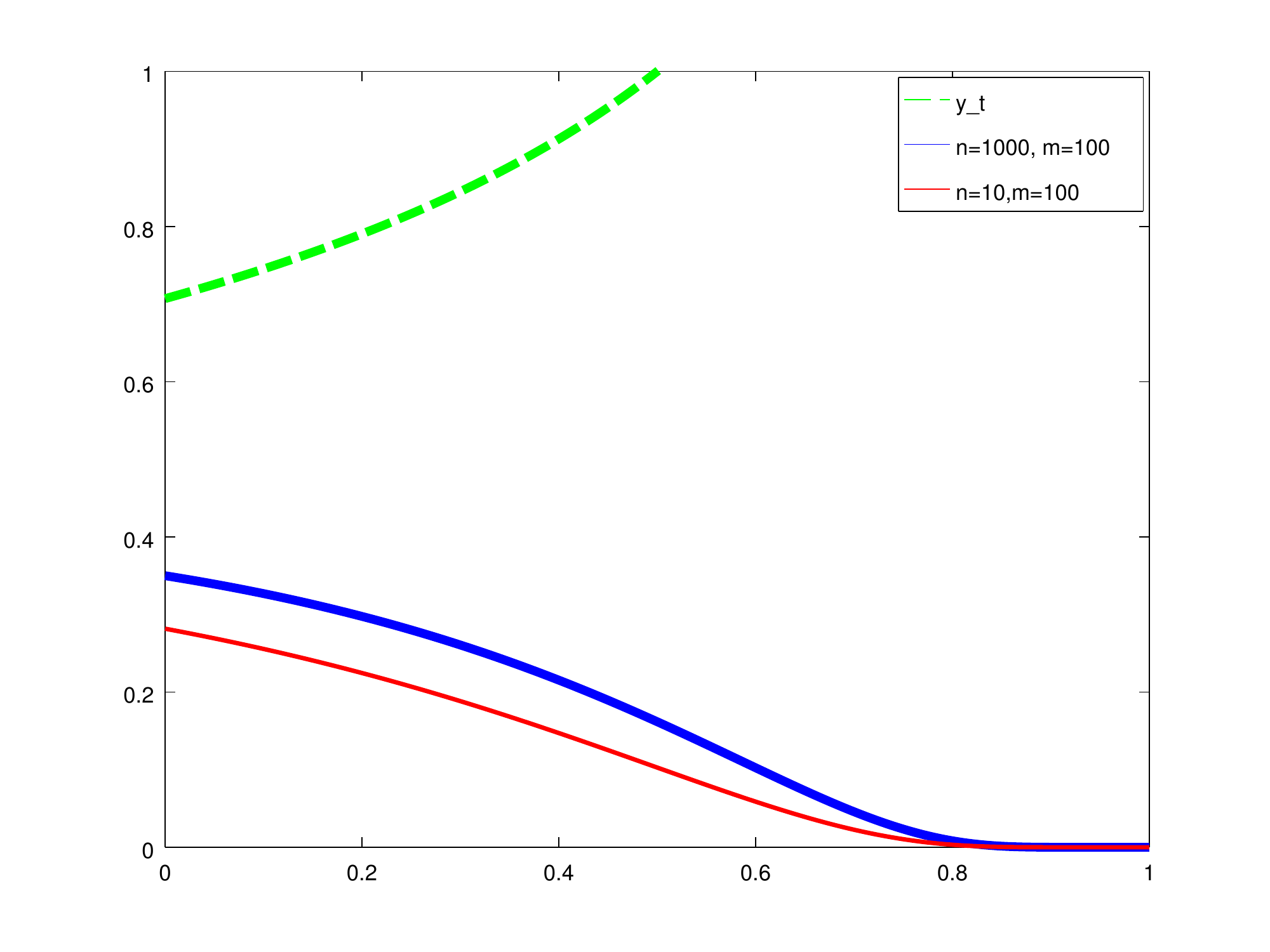}
\end{center}
\caption{\hspace{0.25cm}On the left, the graph of
$u_{m,n}$ with $m=100$ and $n=50$; on the right, the graph of $u_{m,n}$
over $x=1.5$ for $m=100$, $n=10$ (thin) and $n=1000$ (thick),
and $y_t$ (dashed line). In all computations $L=3$ and $T=1$ }
\label{f:u1}
\end{figure}
\begin{figure}
\begin{center}
\includegraphics[width=0.5\textwidth]{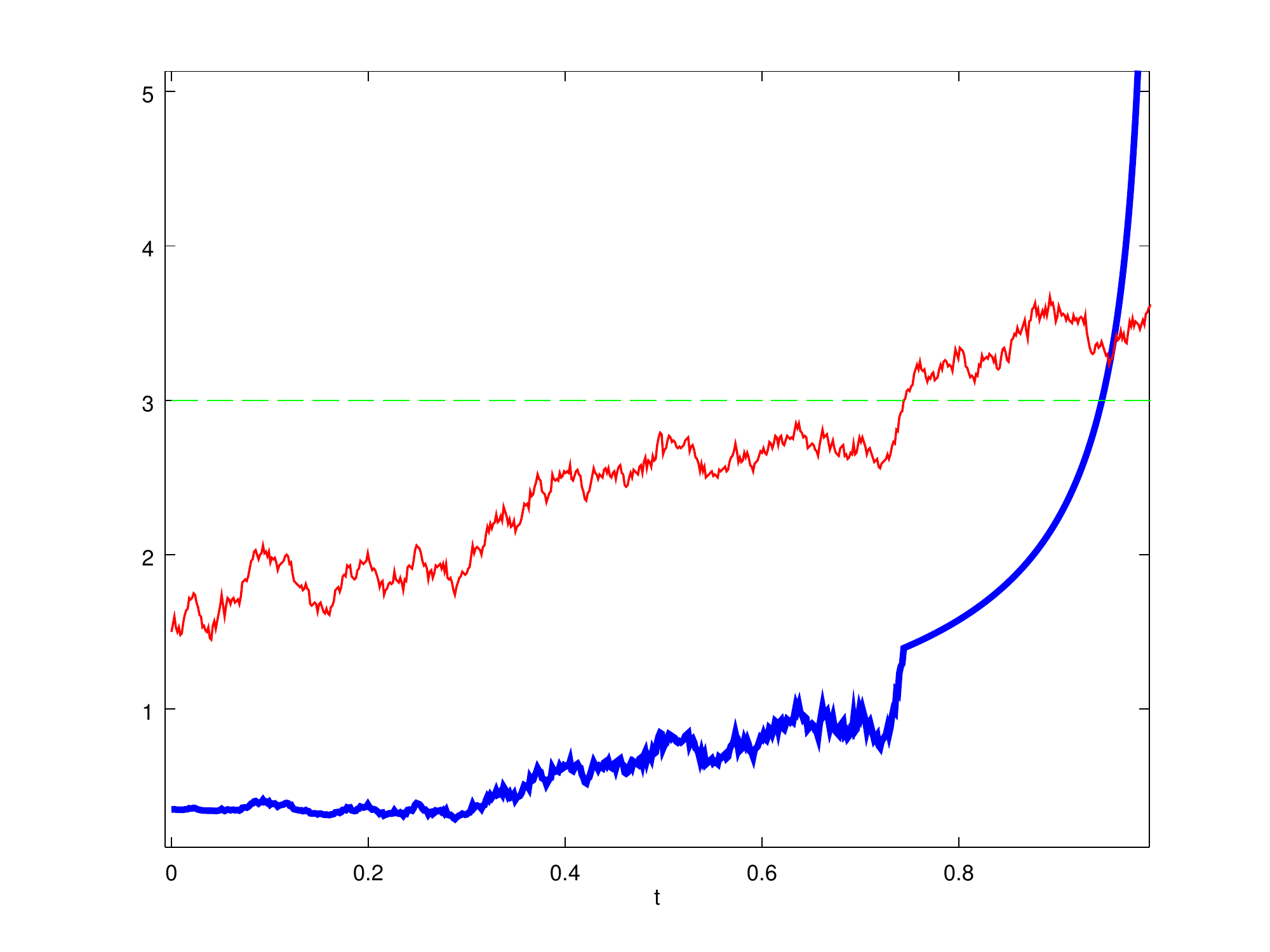}\includegraphics[width=0.5\textwidth]{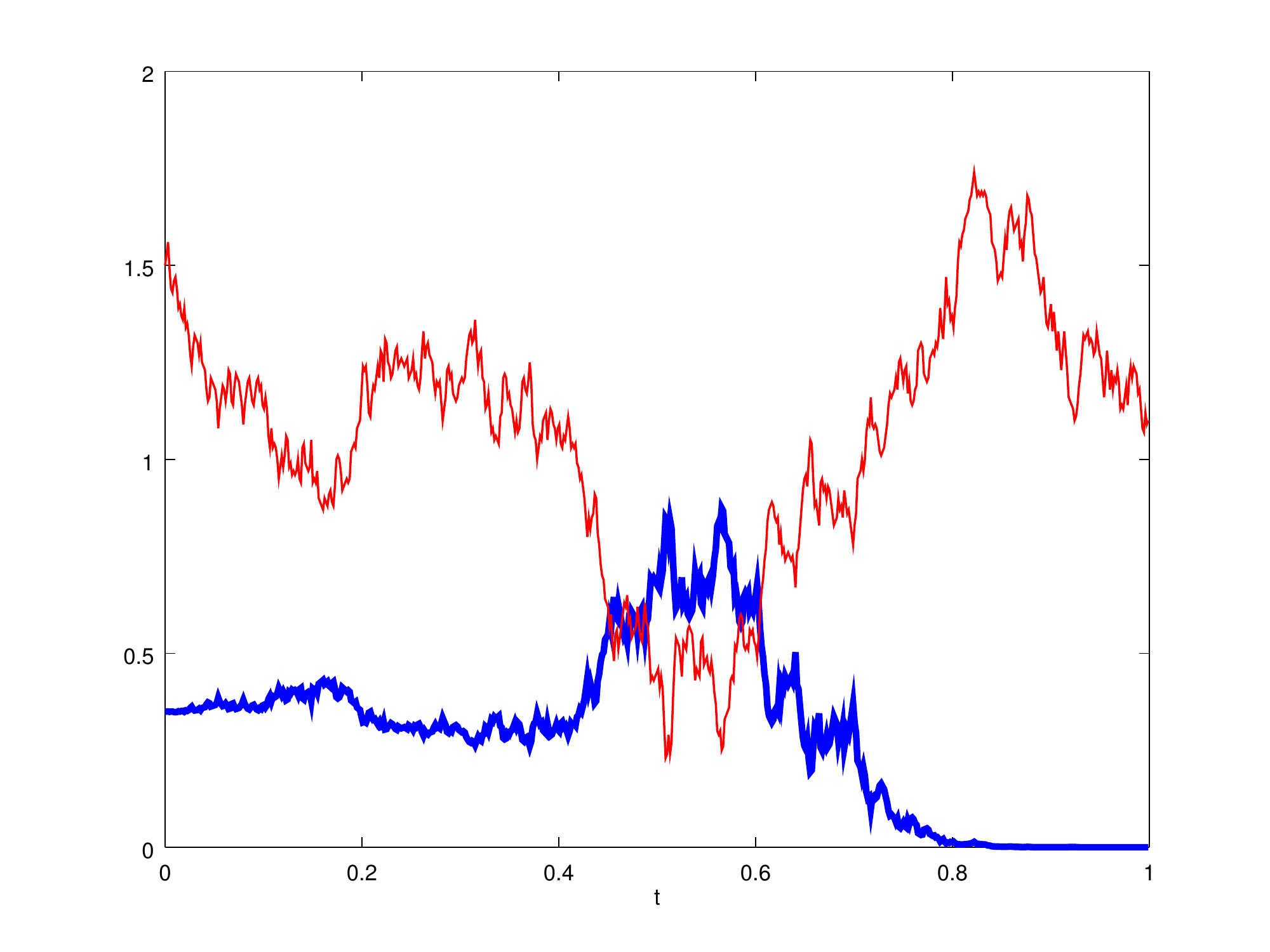}
\end{center}
\caption{Numerically computed trajectories of $W$ (thin
light path) and $Y$  (thick dark)
(left with explosion, right without); $Y$ is computed
using \eqref{e:defYZ} with $u_{m,n}$ approximating $u$
with $m=100$ and $n=1000$; $L=3$ and $T=1$
 \label{f:Y1}}
\end{figure}

\section{The case $\xi = \infty \cdot \mathbf{1}_{B(m,r)}$}\label{s:secondtermcond}

Our results for the  terminal condition $\xi = \infty \cdot \mathbf{1}_{B(L/2,L/2)}$ parallel those
for
$\xi = \infty \cdot \mathbf{1}_{B(L/2,L/2)^c}$, 
with two {differences: 
1) we need to replace the
upperbound $v_0$ of subsection \ref{ss:heat} 
with a new upperbound $\bar{u}$} and
2) $q>1$ is enough
for the existence of solutions;
see the discussion below
for more on these changes.
\begin{theorem} \label{thm:second_case}
Suppose $q >1.$
There exists a function $\bar{v}$ which is $C^\infty$ in the $x$ variable 
and $C^1$ in the $t$ variable
and continuous on $\bar{D} \setminus \{(L,T),(0,T)\}$ and which solves the PDE \eqref{e:PDEfirst} with the boundary condition 
\eqref{e:boundary2} (i.e., \eqref{e:boundary2r} below)) such that
{
\begin{enumerate}
\item 
The processes
\begin{equation}\label{e:defYZ2}
Y_t = \begin{cases} \bar{v}(W_t,t)&, t < \tau \wedge T,\\
			 0&, \tau \le t \le T,
	  \end{cases}~~~
Z_t = \begin{cases} \bar{v}_x(W_t,t)&, t < \tau \wedge T,\\
			 0&, \tau \le t \le T.
	  \end{cases}
\end{equation}
solve the BSDE (\ref{e:SDE}, \ref{e:terminalcondition}) 
with
$\xi = \infty \cdot \mathbf{1}_{B(L/2,L/2)}$, and 
in particular, $Y$ is continuous on $[0,T]$,
\item 
$(\Ymin,\Zmin)= (Y,Z)$ and \eqref{e:cont_terminalcondition} hold.
\end{enumerate}
}
\end{theorem}

The steps for the proof of Theorem \ref{thm:second_case} 
apply verbatim to the current case except for the construction
of the solution of the PDE; for this reason we only give an outline
and point out the necessary changes.
Breaking as we do in Section \ref{ssect:reduc} the BSDE 
into the intervals $[0,\tau \wedge T]$ and $[\tau \wedge T, T]$, this case can be reduced to the solution of the PDE \eqref{e:PDEfirst} now with the boundary condition
\begin{equation}\label{e:boundary2r}
V(0,t)  = V(L,t) = 0, t \in [0,T],~~
V(x,T) = \infty, 0  < x < L.
\end{equation}
The construction given in Section \ref{ssect:solvepde} for the
PDE \eqref{e:PDEfirst} and the  
boundary condition \eqref{e:boundary1} allow one to solve the same
PDE 
now with the boundary condition \eqref{e:boundary2r} except for the 
differences pointed out above:
in the present case we no longer have the upperbound $v_0$ 
to serve as an upperbound
in convergence and continuity arguments.
The role of $v_0$ will now be played
by the limit $\bar{u}$ of a decreasing sequence of solutions of
\eqref{e:PDEfirst}.
And because we no longer need $v_0$ we no longer need the assumption $q >2$
and can work with $q>1$.
The details are given in the outline below:
\begin{enumerate}
\item 
First proceed as in Section \ref{ssect:solvepde}, Lemma \ref{l:defu}, to construct a classical
solution $\bar{u}_n$  to \eqref{e:PDEfirst} on $[0,L]  \times [0, T-1/n]$ with 
the boundary condition
\begin{equation*}
V(0,t)  = V(L,t) = 0, t \in [0,T-1/n],~~
V(x,T-1/n) = y_{T-2/n},
\end{equation*}
 $0  < x < L$, continuous on $[0,L] \times [0,T-1/n] - \{(0,T-1/n),(L,T-1/n)\}$
satisfying
the expectation representations (of type \eqref{e:defu} and \eqref{e:defu_equiv}):
\begin{equation} \label{e:def_u_bar}
\bar{u}_n(x,t) = {\mathbb E}_{x,t}\left[e^{-\int_t^{T-1/n} \bar{u}_n^{q-1}(W_s,s)ds} y_{T-2/n}  
\mathbf{1}_{\{\tau \geq T-1/n\}}\right]
\end{equation}
or equivalently
\begin{equation}\label{e:equiv_def_u_bar}
\bar{u}_n(x,t) = {\mathbb E}_{x,t}\left[-\int_t^{\tau \wedge (T-1/n) } \bar{u}_n^q(W_s,s)ds +
y_{T-2/n}  \mathbf{1}_{\{\tau \geq T-1/n\}}  \right ],
\end{equation}
for $(x,t) \in [0,L]  \times [0, T-1/n]$. 
\item By Lemma \ref{l:monotonicity} for any $(x,t) \in [0,L]  \times [0, T-1/n]$, $\bar{u}_n(x,t) \leq y_t$. Hence for $n_1 < n_2$ and for any $x \in [0,L]$, $\bar u_{n_2}(x,T-1/n_1) \leq y_{T-2/n_1}$. Again by comparison principle
(Lemma \ref{l:monotonicity}), 
for any $(x,t) \in [0,L]  \times [0, T-1/n_1]$,  $\bar u_{n_2}(x,t) \leq \bar u_{n_1}(x,t)$. Then, for fixed $(x,t) \in [0,L]\times[0,T)$,
$(x,t)$ will be in the domain of $\bar u_n(x,t)$ for $ n \ge n_0$ for
some $n_0$ and the sequence $\{u_n(x,t), n \ge n_0\}$  will be decreasing
in $n$; call its limit $\bar{u}(x,t)$, i.e.,
\[\bar u_{n}(x,t) \searrow  \bar u(x,t), \quad n\to +\infty.\]
The representations \eqref{e:def_u_bar} and \eqref{e:equiv_def_u_bar} and the dominated convergence theorem
imply
\begin{align} \label{e:DPEforubar1}
\bar{u}(x,t) &= {\mathbb E}_{x,t}\left[e^{-\int_t^{\tau \wedge r} \bar{u}^{q-1}(W_s,s)ds} \bar u(W_r,r) \right],\\
\label{e:DPEforubar2}
\bar{u}(x,t) &= {\mathbb E}_{x,t}\left[-\int_t^{\tau \wedge r} \bar{u}^q(W_s,s)ds +
\bar u (W_r,r)  \right ],
\end{align}
for any $t < r < T$ and any $x\in [0,L]$.
Moreover we have 
$$\forall (x,t)\in [0,L]\times [0,T), \quad 0\leq \bar u(x,t) \leq y_t.$$
\item The sufficient differentiability of $\bar{u}$ is proved exactly as in the proofs of Lemma \ref{l:regularityexistence}.
This implies (via It\^o's formula) that $\bar{u}$ solves \eqref{e:PDEfirst}.
\item 
{Next we construct an increasing approximating sequence}
$\bar{v}_n$ by solving the PDE \eqref{e:PDEfirst} on $[0,L] \times [0, T]$ with the boundary condition
\begin{equation*}
V(0,t)  = V(L,t) = 0, t \in [0,T],~~
V(x,T) = y_{T-1/n}, 0  < x < L,
\end{equation*}
and continuous on  $\bar{D} \setminus \partial B.$
\item  Lemma \ref{l:monotonicity} implies  that the sequence $\bar{v}_n$ is increasing.
Define $\bar{v} \doteq \lim_{n\rightarrow \infty} \bar{v}_n.$ 
\item
Lemma \ref{l:monotonicity} and the fact that $(x,t) \mapsto y_{t-1/n}$ solves \eqref{e:PDEfirst}
 imply that $\bar v_n(\cdot,T-1/n) \leq y_{T-2/n}$, which, along with Lemma \ref{l:monotonicity} and
the definition of $\bar{u}_n$ imply
\[
\bar{v}_n \le \bar{u}_n,
\]
from which 
\begin{equation*}
\bar{v} \le \bar{u}
\end{equation*}
follows.
Arguments in subsection \ref{ssect:solvepde}
now applied to $\bar{v}$ (with $\bar{u}$ providing the dominating function)
imply that $\bar{v}$ has representations of the form \eqref{e:DPEforubar1}
 and \eqref{e:DPEforubar2}, is infinitely differentiable in the $x$ variable
and continuously differentiable in the $t$ variable  on $(0,L)\times (0,T)$ and that it solves \eqref{e:PDEfirst}.
\item The functions $\bar{v}$ and $\bar{u}$ 
both satisfy the boundary condition \eqref{e:boundary2r} by definition.
It remains to show that they are
continuous on $\partial D \setminus \partial B$.
The continuity of $\bar{u}$ on the lateral boundary $\{(0,t), (L,t), t < T \}$ follows from $0 \le \bar{u} \le \bar{u}_n$ and the 
continuity of $\bar{u}_n$ on the same boundary. $0 \le \bar{v} \le \bar{u}$
implies then the continuity of $\bar{v}$ on the lateral boundary.
The continuity of $\bar{v}$
along $B$ follows from the continuity of $\bar{v}_n$ along the same boundary and $\bar{v} \ge \bar{v}_n.$
This and $\bar{u} \ge \bar{v}$ finally imply the continuity of $\bar{u}$ along $B$.
\end{enumerate}
The above algorithm gives us two classical solutions $\bar{u},\bar{v}$ of the PDE \eqref{e:PDEfirst}
and the boundary condition \eqref{e:boundary2r}
satisfying $\bar{u} \ge \bar{v}.$
In \eqref{e:defYZ2} we use the smaller of these solutions to define our solution of the 
BSDE (\ref{e:SDE},\ref{e:terminalcondition}) with
$\xi = \infty \cdot \mathbf{1}_{B(m,r)}$.
That  $(Y,Z)$ thus defined satisfies 
(\ref{e:SDE},\ref{e:terminalcondition}) as well as the proof of continuity of $Y$ on $[0,T]$
proceed exactly as in the proof of Proposition \ref{p:BSDEsol}. The proof that $(Y,Z) = (\Ymin,\Zmin)$ proceeds
as in the proof of Proposition \ref{p:YeqYmin} given in subsection \ref{ss:completion} and follows from
$\bar{v}_n \nearrow \bar{v}.$

We illustrate the computations above with several numerical examples
in 
Figures \ref{f:u2} and \ref{fig:Y_sec_case}. The
left side of Figure \ref{f:u2} shows the graph of $\bar{u}_{50}$,
computed numerically using finite differences; the right
side of the same figure shows the graphs of $\bar{u}_{5}(1,t)$ and
$\bar{u}_{50}(1,t)$ and $y_t$. Figure \ref{fig:Y_sec_case} shows two sets
of sample paths of $W$ and $Y$ with $W_0 = L/2=1$ and
 where $Y$ is approximated by
$\bar{u}_{50}(W_t,t)$ for $t < \tau$; in all computations
$L=2$ and $T=1.$

\begin{figure}[h]
\begin{center}
\includegraphics[width=0.5\textwidth]{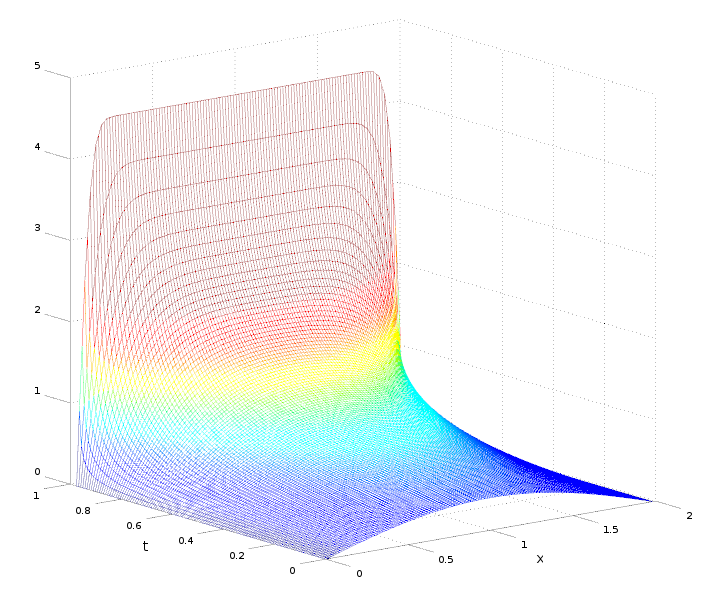}\includegraphics[height=5cm]{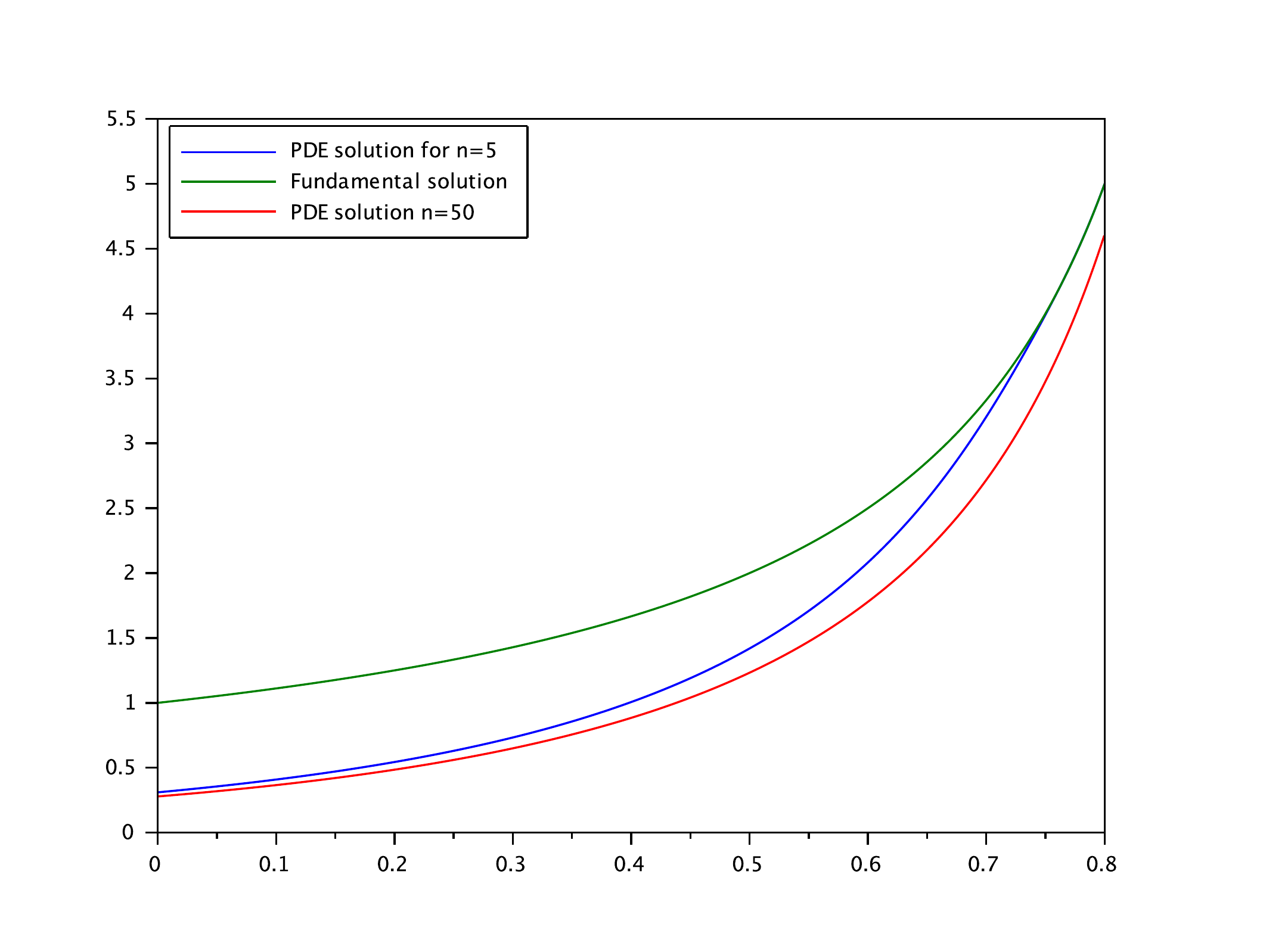}
\end{center}
\caption{\hspace{0.25cm} On the left, graph of $\bar{u}_{50}$;
on the rights graphs of $\bar{u}_5(1,t)$, $\bar{u}_{50}(1,t)$
and $y_t$, $t\in [0,1]$; $T=1$ and $L=2$ \label{f:u2}}
\end{figure}
\begin{figure}[h]
\begin{center}
\includegraphics[height=4.5cm]{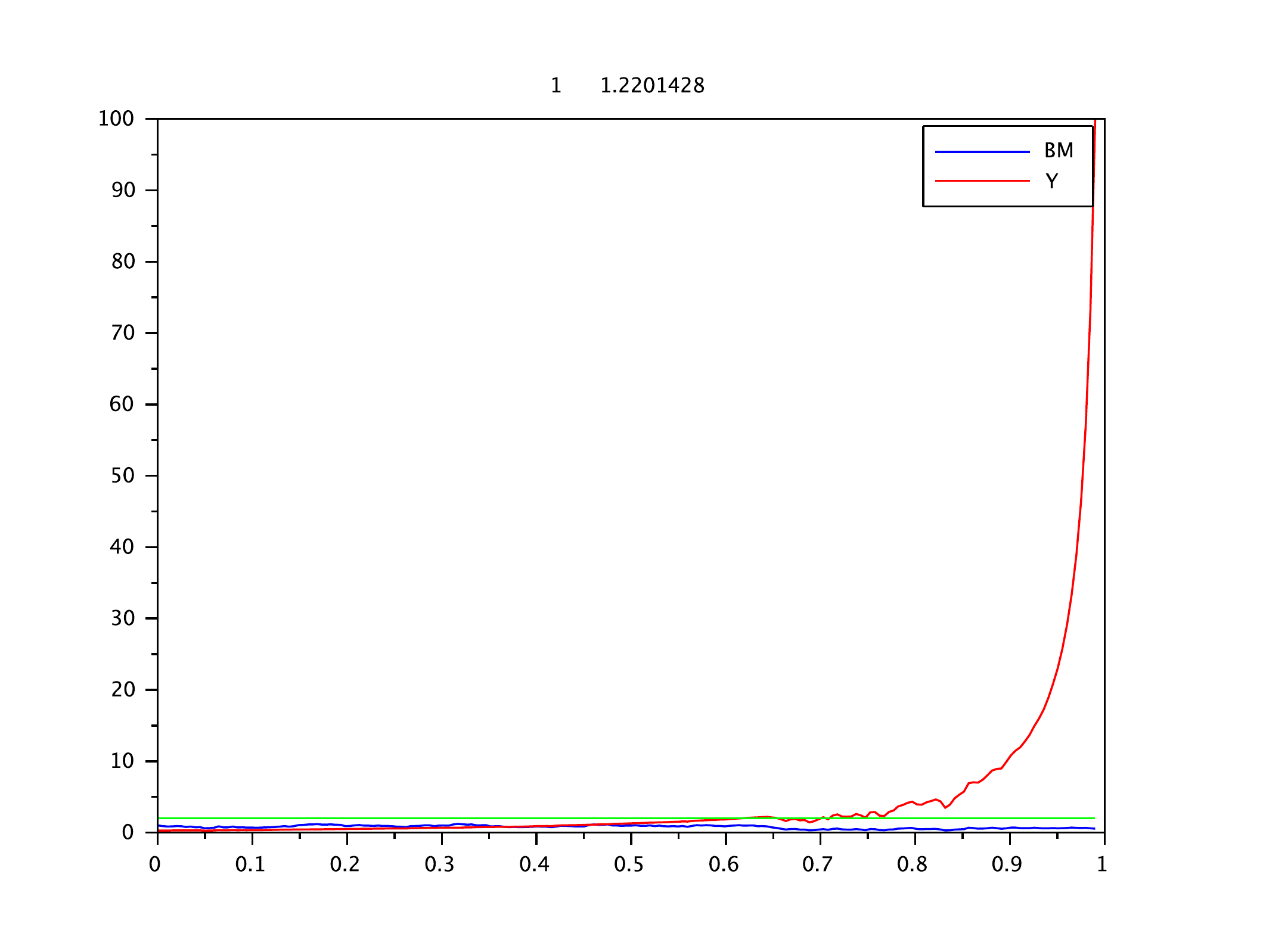}\includegraphics[height=4.5cm]{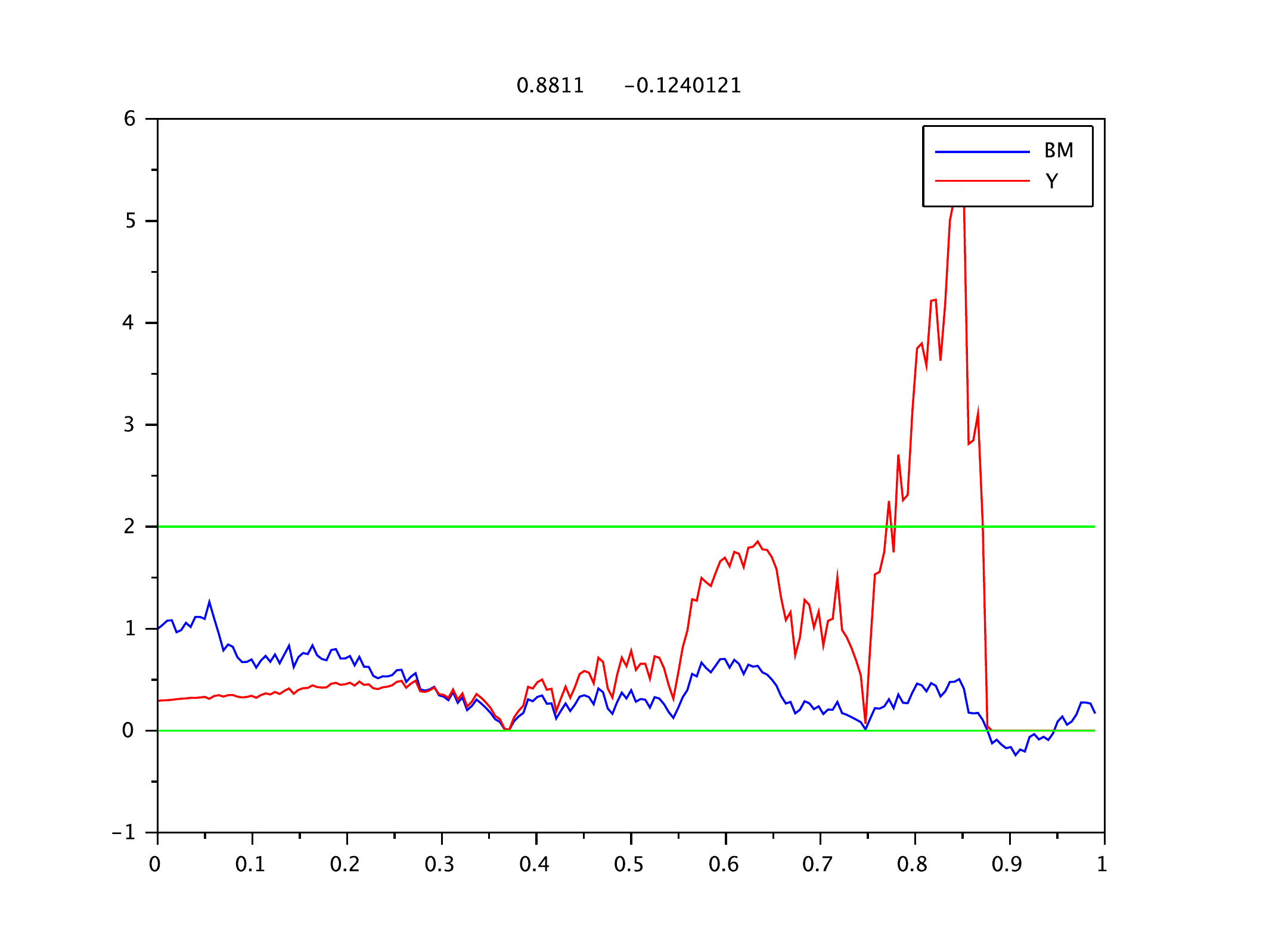}
\end{center}
\caption{\hspace{0.25cm}Two trajectories of $W$ and $Y$ (left with explosion, right without). \label{fig:Y_sec_case}}
\end{figure}

\section{The Control Interpretation}\label{s:control}

We next point out a control interpretation of the BSDE
(\ref{e:SDE},\ref{e:terminalcondition}) or more precisely
of the BSDEs (\ref{e:SDE},\ref{e:cont_terminalcondition})
and 
(\ref{e:SDE},\ref{e:upperbound}).
We consider the case of a general $\mathscr F_T$-measurable terminal condition $\xi$ possibly taking the value $+\infty$ with positive probability. We use this connection to a stochastic control problem to derive a sufficient (Lemma \ref{l:probreduction}) and a necessary (Lemma \ref{l:neccond}) condition for the continuity $\lim_{t\to T}\Ymin_t=\xi$ of $\Ymin$ at $T$.
Finally, we apply our findings from Sections \ref{s:first_case} and \ref{s:secondtermcond} to derive estimates about the limiting behavior of conditional probabilities $\mathbb P_t[A]$ as $t\to T$ 
for
$A = B(m,r)^c \in {\mathscr F}_T$
and
$A = B(m,r) \in {\mathscr F}_T$
(Corollary \ref{c:limitcondprob}).

Let us assume $p> 1$ and as before
$q$ denotes its H\"older conjugate; 
for an arbitrary  $ \xi \in {\mathscr F}_T$, $\xi \ge 0$,
consider the stochastic optimal control problem
\begin{align}\label{e:control0}
V(c,t,\omega) &\doteq \essinf_{\alpha \in \mathcal A(t,c)} {\mathbb E} \left[
(p-1)^{p-1}\int_t^T |\alpha_s|^{p}ds + |C^\alpha_T|^{p} \xi \bigg| {\mathscr F}_t \right],\\
C^\alpha_u & \doteq c +\int_t^u \alpha_s ds, u\in [t,T], \notag 
\end{align}
$ t\in [0,T]$, $c\in {\mathbb R}$,
where the set of admissible controls $\mathcal A(t,c)$ consists of all progressively measurable processes $\alpha$ such that $\alpha \in L^1(t,T)$ $\mathbb P$-a.s
and we assume $0 \cdot \infty = 0$. From the verification Theorems \cite[Theorem 1.3]{ankirchner2014bsdes} or \cite[Theorem 3]{krus:popi:16}, we know that 
\begin{equation} \label{e:link_control_BSDE}
V(c,t,\omega) = |c|^p \Ymin_t
\end{equation}
where $\Ymin$ is the minimal super-solution of the BSDE 
(\ref{e:SDE}, \ref{e:terminalcondition})  constructed in these works%
\footnote{
In \cite{krus:popi:16} only the weak terminal constraint $\liminf_{t\to T}\Ymin_t\ge \xi$ was established. In the present setting, it follows from the results in \cite{popier2006backward} that actually the limit $\lim_{t\to T}\Ymin_t$ exists and consequently $\lim_{t\to T}\Ymin_t\ge \xi$ holds.}. 
In particular,
\begin{equation*}
\lim_{t\rightarrow T} V(c,t,\omega) \ge |c|^p \xi(\omega)
\end{equation*}
holds $\mathbb P$-a.s.
Moreover, an optimal control for \eqref{e:control0} is given by $\alpha_s^* = - (q-1) C^{\alpha^*}_s |Y_s|^{q-1}$,
and thus
\[
C^{\alpha^*}_u = c \exp \left[ -(q-1) \int_t^u \left( Y_s \right)^{q-1} ds\right],
\]
for $t \leq u \leq T$.
The link \eqref{e:link_control_BSDE} between the value function $V$ 
and $\Ymin$ will give two results concerning the continuity of $\Ymin$ in a general setting. First, 
we show $|\lim_{t\to T}\Ymin_t| < \infty$ already implies continuity.

\begin{lemma}\label{l:probreduction}
Let $\Ymin$ be the minimal super-solution of the BSDE 
{\rm (\ref{e:SDE},\ref{e:terminalcondition})}.
 Then the following implication holds for almost all $\omega \in \Omega$. If 
\begin{equation*} 
\lim_{t\to T}\Ymin_t(\omega)<\infty,
\end{equation*}
 then the path $\Ymin(\omega)$ is continuous at $T$, i.e.\ it holds that $\lim_{t\to T}\Ymin_t(\omega)=\xi(\omega)$.
\end{lemma}

\begin{proof}
Let $C_t=\exp\left(- (q-1)\int_0^t (\Ymin_s)^{q-1}ds\right)$ (which is the optimal control in \eqref{e:control0} for $t=0$ and $c=1$). Since $\Ymin$ is nonnegative, it follows that $C$ is continuous at $T$: $\lim_{t\to T}C_t=C_T \in [0,1]$. We know from the analysis of the control problem (see \cite{ankirchner2014bsdes}, proof of Theorem 4.2) 
that
\begin{equation}\label{e:controleq}
\Ymin_t C_t^{p}=\mathbb{E} \left[\left. (p-1)^{p-1} \int_t^T |\alpha_s|^p ds+ \xi |C_T|^p\right| {\mathscr F}_t \right].
\end{equation}
The right side of \eqref{e:controleq} decomposes into a semimartingale $M+A$ with
$$M_t=\mathbb{E} \left[\left.  (p-1)^{p-1}\int_0^T |\alpha_s|^p ds+  \xi|C_T|^p \right|{\mathscr F}_t \right]$$
and
$$A_t=- (p-1)^{p-1} \int_0^t  |\alpha_s|^p ds.$$
Since $Y_0=\mathbb{E}[M_T]<\infty$, the process $M$ is a true martingale. In particular, $M$ is continuous.  $A$ is continuous by the
fundamental theorem of calculus.
Therefore, the right side of \eqref{e:controleq} converges to $M_T+A_T= \xi|C_T|^p$ as $t\nearrow T$, i.e.
$$\lim_{t\nearrow T}\Ymin_tC_t^p= \xi|C_T|^p.$$
Therefore, on the set $\{C_T>0\}$ we have $\lim_{t\nearrow T}\Ymin_t= \xi$. Moreover, the definition of $C$ implies that if $C_T(\omega)=0$, then $\lim_{t\to T}\Ymin_t(\omega)=\infty$ and consequently $\{\lim_{t\to T}\Ymin_t<\infty\}\subseteq \{C_T>0\}$. This completes the proof.
\end{proof}

\begin{remark}{\em
The identity \eqref{e:controleq} is equivalent to the representation \eqref{e:defu} in Lemma \ref{l:defu} (replace $\Ymin_t$ by $u(W_t,t)$ and 
take expectation).} 
\end{remark}

The next result gives a necessary condition for continuity. We use the shorthand notation $\mathbb P_t[\xi=\infty]=\mathbb E[\mathbf{1}_{\{\xi=\infty\}}|{\mathscr F}_t]$.
\begin{lemma}\label{l:neccond}
Let $\Ymin$ be the minimal super-solution of the BSDE {\rm (\ref{e:SDE},\ref{e:terminalcondition})} and suppose that continuity condition \eqref{e:cont_terminalcondition} holds for $\Ymin$. Then we have a.s. on $\{\xi<\infty\}$
$$\sup_{t\in[0,T]} \frac{\mathbb{P}_t[\xi=\infty]}{(T-t)^{p-1}}<\infty.$$ 
\end{lemma}
\begin{proof}
For $t<T$ let $\alpha \in \mathcal A(t,1)$ be an arbitrary strategy with associated position path $C^\alpha$ that has finite costs:
$$ {\mathbb E} \left[
(p-1)^{p-1}\int_t^T |\alpha_s|^{p}ds + |C^\alpha_T|^{p} \xi \bigg| {\mathscr F}_t \right]<\infty.$$
Take for example the optimal strategy.
 Then we have
\begin{eqnarray*}
&& {\mathbb E} \left[
(p-1)^{p-1}\int_t^T |\alpha_s|^{p}ds + |C^\alpha_T|^{p} \xi \bigg| {\mathscr F}_t \right] \\
&& \qquad \ge (p-1)^{p-1} \mathbb{E}\left[\left.\mathbf{1}_{\{C^\alpha_T=0\}}\int_t^T|\alpha_s|^p ds\right|{\mathscr F}_t \right].
\end{eqnarray*}
Jensen's inequality yields $\int_t^T|\alpha_s|^p\ge \frac 1{(T-t)^{p-1}}$ for every path satisfying $C^\alpha_T=0$. Moreover, since $\alpha$ has finite costs, it holds that $\{\xi=\infty\}\subseteq{\{C^\alpha_T=0\}}$. This implies 
\[
{\mathbb E} \left[
(p-1)^{p-1}\int_t^T |\alpha_s|^{p}ds + |C^\alpha_T|^{p} \xi \bigg| {\mathscr F}_t \right]\ge (p-1)^{p-1} \frac{\mathbb{P}_t[\xi=\infty]}{(T-t)^{p-1}}. 
\]
Since the right side of the above display does not depend on the control $\alpha$, we use \eqref{e:link_control_BSDE} to arrive at
$$
\Ymin_t\ge (p-1)^{p-1} \frac{\mathbb{P}_t[\xi=\infty]}{(T-t)^{p-1}}.
$$
Since $\lim_{t\to T}\Ymin_t<\infty$ if $\xi<\infty$, this yields the claim.
\end{proof}

Lemma \ref{l:neccond} combined with our results from Sections \ref{s:first_case} and \ref{s:secondtermcond} allows to derive estimates on the speed of convergence $\lim_{t\to T}\mathbb P_t[A] \to 0$ on $A^c$ for 
$A = B(m,r)$ and $A=B(m,r)^c.$ This is subject of the next corollary. 

\begin{corollary}\label{c:limitcondprob}
Let $m\in \mathbb R$ and $r\in (0,\infty)$. Then on $B(m,r)$ it holds for all $p\in (1,2)$ that
$$\sup_{t\in[0,T]} \frac{\mathbb{P}_t[B(m,r)^c]}{(T-t)^{p-1}}<\infty$$
and on $B(m,r)^c$ it holds for all $p>1$ that
$$\sup_{t\in[0,T]} \frac{\mathbb{P}_t[B(m,r)]}{(T-t)^{p-1}}<\infty.$$
\end{corollary}
\begin{proof}
For the first result set $\xi=\infty \cdot \mathbf{1}_{B^c(m,r)}$ and
let $(\Ymin,\Zmin)$ denote the minimal super-solution of (\ref{e:SDE},\ref{e:terminalcondition}) and let $(Y,Z)$ denote the solution of (\ref{e:SDE},\ref{e:terminalcondition}) constructed in Proposition \ref{p:BSDEsol}. In particular, it holds that $\lim_{t\to T}Y_t=\xi$. Minimality of $\Ymin$ and $\lim_{t\to T} \Ymin_t\ge \xi$ imply that $\lim_{t\to T}\Ymin_t=\xi$. The first result then follows from Lemma \ref{l:neccond} (observe that $q>2$ implies that $p<2$).
The second result follows from the same argument and Lemma \ref{l:neccond}, 
this time used with the results in Section \ref{s:secondtermcond}.
\end{proof}

\section{Conclusion}

Let us comment on several direct extensions and possible future work.
The extension of the boundary condition $\xi= \infty \cdot \mathbf{1}_{B^c}$ to
  $\infty \cdot \mathbf{1}_{B^c} + g(W_T) \mathbf{1}_{B}$ for $g$ such
that $\mathbb{E}(|g(W_T)| \mathbf{1}_B) < \infty$ requires only
that we change the terminal condition \eqref{e:boundary1} 
\begin{equation*}
V(0,t)  = V(L,t) = y_t, \ t \in [0,T],~~
V(x,T) = g(x), \ 0  < x < L.
\end{equation*}
Simple modifications of the argument of Section \ref{s:first_case} would suffice
to deal with this change.
Generalizing  the terminal condition $\xi =\infty \cdot \mathbf{1}_{B}$ to
 $\infty \cdot \mathbf{1}_{B} + g(W_T) \mathbf{1}_{B^c}$ for $g$ such that $\mathbb{E}(|g(W_T)| \mathbf{1}_B) < \infty$ requires the solution of two PDE:
one must first solve \eqref{e:PDEfirst} over the domain $ {\mathbb R} \times [0,T]$ 
where $g$ serves as terminal condition on the terminal boundary of this domain. 
The value of the solution
on $ S = \{(L,t), t \in [0,T] \} \cup \{ (0,t), t \in [0,T] \}$ will 
then serve as lateral boundary
condition for the PDE \eqref{e:PDEfirst} on $D$. 

A further generalization involves changing the definition of the set $B$ to $\{\omega -r < X_t(\omega) -c < r, t \in [0,T]\}$
where $X$ is an SDE driven by $W$;
this generalization would require to modify the second derivative term in \eqref{e:PDEfirst} to the infitesimal
generator of $X$. Further generalizations can consider the case when $X$ is an SDE with
jumps or a doubly stochastic process, which may require further arguments and ideas.
The treatment of these extensions may also be taken up in future work. 

For the case when $\xi = \xi_1 =  \mathbf{1}_{B^c}$ our arguments depended on $q > 2$, which implied 
${\mathbb E}[y_\tau \mathbf{1}_{\{\tau < T\}}] < \infty.$
The work of Marcus \& V\'eron \cite{marc:vero:01} and numerical computations suggest that even when 
$q \in (1,2]$ the PDE \eqref{e:PDEfirst} and the boundary condition \eqref{e:boundary1} have a smooth solution.  
Future work can also try to treat the terminal condition $\xi_1$ with $q \in (1,2].$

The single space dimension that we have treated in the present work
simplified our existence and smoothness arguments for the
solutions of the PDE we have studied. Their extension to higher dimensions could also be the
subject of future work. In this, a possible approach is, as hinted at in the introduction, to develop 
arguments for our PDE problems starting from results of \cite{marcus1999initial}.

Finally, from an applied perspective, we think that it would be of interest
to study the implications of the results
of the current work for the portfolio liquidation problem mentioned in 
Section \ref{s:control}.

\def\cprime{$'$} \def\cprime{$'$}
\providecommand{\bysame}{\leavevmode\hbox to3em{\hrulefill}\thinspace}
\providecommand{\MR}{\relax\ifhmode\unskip\space\fi MR }
\providecommand{\MRhref}[2]{%
  \href{http://www.ams.org/mathscinet-getitem?mr=#1}{#2}
}
\providecommand{\href}[2]{#2}

\end{document}

%% file: thedomain.tex
\begin{picture}(0,0)%
\includegraphics{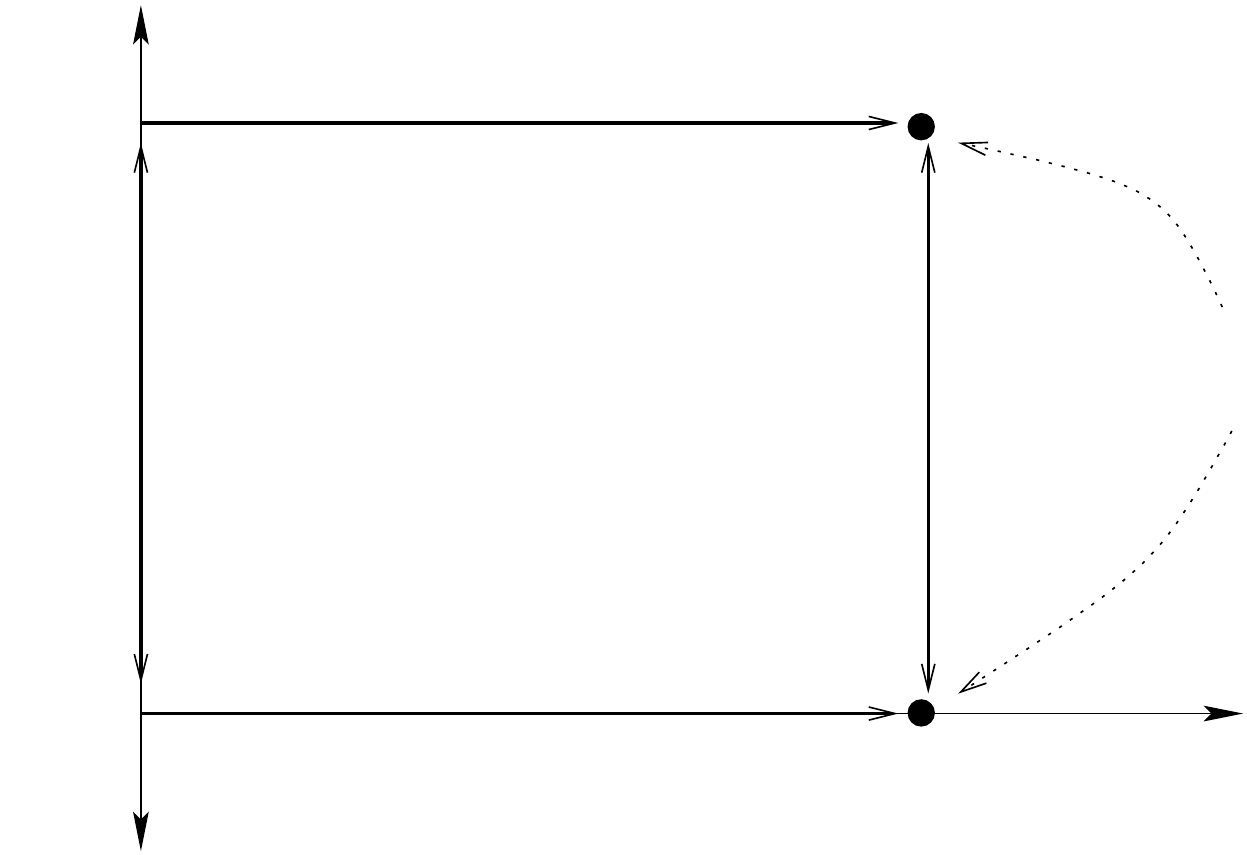}%
\end{picture}%
\setlength{\unitlength}{4144sp}%
\begingroup\makeatletter\ifx\SetFigFont\undefined%
\gdef\SetFigFont#1#2#3#4#5{%
  \reset@font\fontsize{#1}{#2pt}%
  \fontfamily{#3}\fontseries{#4}\fontshape{#5}%
  \selectfont}%
\fi\endgroup%
\begin{picture}(5697,3894)(1156,-4753)
\put(3511,-1321){\makebox(0,0)[lb]{\smash{{\SetFigFont{12}{14.4}{\rmdefault}{\mddefault}{\updefault}{\color[rgb]{0,0,0}$S$}%
}}}}
\put(3511,-4021){\makebox(0,0)[lb]{\smash{{\SetFigFont{12}{14.4}{\rmdefault}{\mddefault}{\updefault}{\color[rgb]{0,0,0}$S$}%
}}}}
\put(1531,-1996){\makebox(0,0)[lb]{\smash{{\SetFigFont{12}{14.4}{\rmdefault}{\mddefault}{\updefault}{\color[rgb]{0,0,0}$B_0$}%
}}}}
\put(1936,-1051){\makebox(0,0)[lb]{\smash{{\SetFigFont{12}{14.4}{\rmdefault}{\mddefault}{\updefault}{\color[rgb]{0,0,0}$x$}%
}}}}
\put(1891,-2671){\makebox(0,0)[lb]{\smash{{\SetFigFont{12}{14.4}{\rmdefault}{\mddefault}{\updefault}{\color[rgb]{0,0,0}$t=0$}%
}}}}
\put(6796,-2626){\makebox(0,0)[lb]{\smash{{\SetFigFont{12}{14.4}{\rmdefault}{\mddefault}{\updefault}{\color[rgb]{0,0,0}$\partial B$}%
}}}}
\put(1171,-4156){\makebox(0,0)[lb]{\smash{{\SetFigFont{12}{14.4}{\rmdefault}{\mddefault}{\updefault}{\color[rgb]{0,0,0}$x=0$}%
}}}}
\put(1351,-1456){\makebox(0,0)[lb]{\smash{{\SetFigFont{12}{14.4}{\rmdefault}{\mddefault}{\updefault}{\color[rgb]{0,0,0}$x=L$}%
}}}}
\put(6256,-4021){\makebox(0,0)[lb]{\smash{{\SetFigFont{12}{14.4}{\rmdefault}{\mddefault}{\updefault}{\color[rgb]{0,0,0}$t$}%
}}}}
\put(4861,-3076){\makebox(0,0)[lb]{\smash{{\SetFigFont{12}{14.4}{\rmdefault}{\mddefault}{\updefault}{\color[rgb]{0,0,0}$t=T$}%
}}}}
\put(5446,-2266){\makebox(0,0)[lb]{\smash{{\SetFigFont{12}{14.4}{\rmdefault}{\mddefault}{\updefault}{\color[rgb]{0,0,0}$B$}%
}}}}
\end{picture}%